\documentclass{amsart}
\usepackage[normalem]{ulem}
\usepackage{graphicx}
\usepackage{amssymb,amsmath,amsrefs}
\usepackage[colorlinks=true, linkcolor=red, citecolor=blue, urlcolor=blue]{hyperref}
\usepackage{xcolor}
\usepackage{graphicx}
\usepackage{enumerate}
\usepackage{tikz}
\usetikzlibrary{arrows.meta, positioning}

\usepackage{comment}

\usepackage{soul,color}
\usepackage{array}
\usepackage{wrapfig}
\usepackage{tabularx}
\usepackage{cancel}

\theoremstyle{plain}

\newtheorem{mainthm}{Theorem}

\newtheorem{mainclly}[mainthm]{Corollary}
\newtheorem*{conj*}{Conjecture}
\newtheorem*{cor*}{Corollary}
\newtheorem{theorem}{Theorem}[section]
\newtheorem{thm}[theorem]{Theorem}

\newtheorem{corollary}[theorem]{Corollary}
\newtheorem{lemma}[theorem]{Lemma}

\newtheorem{conj}{Conjecture}

\newtheorem{claim}{Claim}

\theoremstyle{definition}
\newtheorem*{def*}{Definition}
\newtheorem{remark}[theorem]{Remark}

\newtheorem{example}[theorem]{Example}

\newtheorem{definition}[theorem]{Definition}

\newcommand{\SA}{{\mathcal A}}

\newcommand{\SE}{{\mathcal E}}
\newcommand{\SF}{{\mathcal F}}

\newcommand{\SK}{{\mathcal K}}
\newcommand{\SL}{{\mathcal L}}
\newcommand{\SM}{{\mathcal M}}
\newcommand{\SN}{{\mathcal N}}

\newcommand{\SR}{{\mathcal R}}

\newcommand{\SU}{{\mathcal U}}
\newcommand{\SV}{{\mathcal V}}
\newcommand{\SW}{{\mathcal W}}
\newcommand{\SX}{{\mathcal X}}

    \newcommand{\Ga}{\Gamma}
\newcommand{\de} {\delta}

\renewcommand{\epsilon}{\varepsilon}

\newcommand{\ka} {\kappa}

\newcommand{\si} {\sigma}       \newcommand{\Si}{\Sigma}

\newcommand{\Z}{\mathbb{Z}}
\newcommand{\N}{\mathbb{N}}
\newcommand{\R}{\mathbb{R}}
\newcommand{\eps}{\varepsilon}

\newcommand{\su}{\operatorname{Supp}}

\makeatletter
\newcommand{\tpitchfork}{
  \vbox{
    \baselineskip\z@skip
    \lineskip-.52ex
    \lineskiplimit\maxdimen
    \m@th
    \ialign{##\crcr\hidewidth\smash{$-$}\hidewidth\crcr$\pitchfork$\crcr}
  }
}
\makeatother

\DeclareMathOperator{\Sing}{Sing}
\DeclareMathOperator{\supp}{supp}
\DeclareMathOperator{\Ind}{Ind}
\DeclareMathOperator{\Diff}{Diff}

\numberwithin{equation}{section}

\title{Entropy Flexibility of Dynamical Systems}

\author{Alexander Arbieto}
\address[A. Arbieto]{
	Universidade Federal do Rio de Janeiro, Instituto de Matem\'atica,
	Av. Athos da Silveira Ramos, 149,
	21941-909, Rio de Janeiro,
}
\email{arbieto@im.ufrj.br}

\author{Piotr Oprocha}
\address[P. Oprocha]{
	AGH University of Krakow, Faculty of Applied Mathematics,
	al. Mickiewicza 30,
	30-059 Krak\'ow,
	Poland -- and -- National Supercomputing Centre IT4Innovations, University of Ostrava,
	IRAFM,
	30. dubna 22, 70103 Ostrava,
	Czech Republic}
\email{piotr.oprocha@osu.cz}

\author{Elias Rego}
\address[E. Rego]{
	AGH University of Krakow, Faculty of Applied Mathematics,
	al. Mickiewicza 30,
	30-059 Krak\'ow,
}
\email{rego@agh.edu.pl}



\begin{document}

    \begin{abstract}
 Inspired by Katok's intermediate entropy property [Inst. Hautes \'Etudes Sci. Publ. Math. 51 (1980), 137–173], 
 we introduce and study the notion of entropy flexibility for discrete-time and continuous-time dynamical systems. By using renewal systems techniques, we show that this property is present in several classes of systems where any intermediate value of entropy can be attained on a strictly ergodic sub-system. In addition, we prove an  entropy flexibility analogue of Katok’s conjecture: Entropy flexibility is a typical property for vector fields on 3-manifolds and surface diffeomorphisms. 
\end{abstract}

\maketitle

\section{Introduction}

In order to quantify the complexity of a dynamical system, a fundamental quantity is often considered: the \emph{entropy} of the system. If this number is positive, it indicates that the number of orbits with distinguishable behavior grows exponentially fast with time, leading to disorder in dynamical behavior, a phenomenon
referred to as \emph{chaos} by several authors. Entropy can be studied from either a topological or a measure-theoretic perspective. Let us denote by $h_{top}(f)$ the \emph{topological entropy} of a map $f$, and by $h_{\mu}(f)$ the \emph{measure-theoretic entropy} of $f$ with respect to an invariant measure $\mu$. These notions are connected through the \emph{variational principle}:
$$
h_{top}(f) = \sup\{h_{\mu}(f) \; ; \; \text{$\mu$ is an invariant measure for $f$}\}.
$$
Moreover, entropy can also be investigated from a local perspective. The topological entropy is concentrated on the nonwandering set, and if this set decomposes into invariant pieces, then the topological entropy of $f$ equals the supremum of the entropy over these pieces.

However, entropy can also arise in a more homogeneous way. There exist examples of minimal and uniquely ergodic systems with positive entropy, meaning that in a sense, the complexity is uniformly distributed throughout the space; for instance, see \cite{Gr} for early constructions of this type; cf. celebrated Jewett-Krieger theorem. In such systems, there are no proper subsystems carrying smaller entropy.  One could then ask when entropy can change in a ``continuous'' fashion in a given system.
This type of dynamical behavior was first observed by A. Katok in the measure-theoretical setting in \cite{Katok}, where he introduced the following definition.

\begin{definition}\label{def: IEP}
We say that $f$ has the \emph{intermediate entropy property} if for every number $0\leq c<h_{top}(f)$ there exists an ergodic measure $\mu_c$ of $f$ such that $h_{\mu_c}(f)=c$.
\end{definition}

In \cite{Katok}, this property was proved for $C^{1+\alpha}$-diffeomorphisms on surfaces, and the following conjecture was proposed.

\begin{conj*}[Katok]
Every $C^2$-diffeomorphism on a compact Riemannian manifold with dimension greater than one has the intermediate entropy property.
\end{conj*}

The conjecture has been explored by many authors and proved to be true in several scenarios, see, for instance, the works \cites{GSW,Sun1,Sun2,HXX}. However, one can also pose the analogous question from the topological perspective. Motivated by this, we introduce the following definition, which asks for intermediate entropy to be realized simultaneously in both the topological and the measure-theoretical settings.

\begin{definition}\label{def: EF}
A map $f:M\to M$  has \emph{entropy flexibility} if for every $0\leq c< h_{top}(f)$ there is an  ergodic measure $\mu_c$ and a compact and invariant subset $\Lambda_c\subset M$ such that $$h_{top}(f|_{\Lambda_c})=h_{\mu_c}(f)=c.$$
\end{definition}

\begin{remark}
    In \cite{WXZ}, $f$ is called a lowerable system if any intermediate value of topological entropy can be realized by a compact  (not necessarily invariant) subset of $K_c\subset M$. The main result of that paper states that any homeomorphism with finite entropy is lowerable.  Although definition of lowerable system seems similar to Definition~
    \ref{def: EF}, it is much less demanding. When proving that a system is lowerable, the compact set $K_c$ is not required to be invariant,  and in many cases cannot be, in particular in minimal systems. On the other hand, entropy flexibility requires value of topological entropy to be realized by a subsystem of $f$, in particular, minimal system is never entropy flexible.     
\end{remark}

The simplest example of a system that exhibits entropy flexibility is the full-shift on two or more symbols. This can be easily seen; an exemplary proof can be found in the monograph by  Walters~\cite{Walters}. However, the simple argument relies heavily on the fact that concatenations of all words are allowed, a property absent in several subshifts.

Obtaining entropy flexibility directly from the intermediate entropy property is not immediate. Indeed, if we get the measure $\mu_c$, it is not obvious to get the set $\Lambda_c$ directly from $\mu_c$. In fact, when considering $K=\supp(\mu_c)$ the entropy of $f|_{K}$ can be way bigger than $h_{\mu_c}(f)$, as illustrated in the following example:

\begin{example}\label{counter-examples}
In \cite{Her}, the construction of a minimal $C^\infty$-diffeomorphism with positive entropy is presented. By \cite{SunPhd}*{Theorem 3.3.5}, this example has the intermediate entropy property. On the other hand, it cannot have entropy flexibility due to its minimality. 
\end{example}

One can easily formulate the Definitions \ref{def: IEP} and \ref{def: EF} for flows by an obvious adjustment of the terminology. Observe that, by considering the suspension flow of the diffeomorphism in Example \ref{counter-examples} with constant roof function, we can obtain an example of flows with the intermediate entropy property, but without entropy flexibility. By the previous examples, an entropy flexibility analogue of Katok's conjecture is not true for both discrete-time and continuous-time systems. 
Nevertheless, one could still ask whether the set of systems satisfying entropy flexibility is large or negligible. We recall that a property $(P)$  is $C^1$-typical or holds $C^1$-generically if it holds for a residual\footnote{a set containing an intersection of at most countably many open and dense sets}  set of systems.

\begin{conj}\label{conj} Entropy flexibility holds $C^1$-generically for:
\begin{enumerate}
    \item Diffeomorphisms on a compact Riemannian manifold with dimension greater than one.
    \item Flows on a compact Riemannian manifold with dimension greater than two. 
   \end{enumerate}
\end{conj}
It is worth noting that, while entropy flexibility is known for some classes of discrete-time dynamical systems, to the best of our knowledge, no similar results are available for the continuous-time framework, even in the more classical settings of hyperbolic sets or Anosov flows. In this sense, the primary goal of this paper is to fill the gap in results concerning entropy flexibility for flows, as well as to derive consequences for discrete-time systems.

In what follows, we present more precisely our main results. In order to improve the readability of this introduction, we will provide here the basic setting, where our results fit, while the precise definitions for more technical concepts that appear in 
 the present section will be defined 
 when needed in the following sections.
Hereafter $M$ denotes a compact metric space and $f\colon M\to M$ denotes a homeomorphism. We use the symbol $\phi$ to denote any flow over $M$. When $M$ is a smooth Riemannian manifold and $\phi$ is of class $C^1$, we denote by $X_\phi$ or simply just $X$, the vector field induced by $\phi$.

 An immediate consequence of the entropy flexibility for full shifts is that the suspension flow of the full shift map by a \emph{constant} roof function also exhibits flexibility of entropy. However, the same question applies to more general suspensions of general shifts, which is much more delicate. Firstly, entropy flexibility does not always hold for all shifts (in particular, minimal ones). Secondly, constructing subsets with a given value of topological entropy for a flow is considerably more challenging. The main difficulty lies in the relation between the entropy of a measure $\mu$ on the base $\sigma$ and the entropy of its lift $\hat\mu$ for the suspension flow $\phi$, which differs by an integral of the roof function $\rho$, according to Abramov's formula:
$$
h_{\tilde{\mu}}(\phi) = \frac{h_{\mu}(\sigma)}{\int \rho \, d\mu}.
$$
Thus, if $\rho$ is not constant, the integral $\int \rho \, d\mu$ varies with the measure.

In \cite{LSWW}, it was proven that every suspension flow of a subshift of finite type satisfies the intermediate entropy property. As a consequence, the authors showed that a similar result also holds for star flows. However, when attempting to promote their results to obtain entropy flexibility, we face a new challenge. The main difficulty is that, when trying to use metric entropy to construct compact sets with suitable topological entropy, we must simultaneously control the growth of entropy provided by the variational principle applied to the measures $\hat{\mu}$ and the distortion caused by the integral $\int \rho \, d\mu$ in Abramov's formula.

In our first main result, we extend the results from \cite{LSWW}, showing that suspension flows of shifts of finite type have entropy flexibility. An important tool in our approach is approximation built through the shadowing property of the map in the base. For this purpose, we need to deal with more sophisticated coding techniques. Consequently, in addition to immediate improvement, the nature of our techniques allows us to show that not only does entropy flexibility hold, but actually any intermediate value of entropy can be realized by a system with homogeneous dynamics.

\begin{mainthm}
\label{thm: SFT}
Let $(\sigma_A,\Sigma_A)$ be any subshift of finite type and $\rho\colon \Sigma_A\to (0,\infty)$ be any continuous function. Let $\phi\colon \R\times \Lambda\to \Lambda$ be the suspension flow generated by $(\sigma_A,\rho)$. For any $0\leq c< h_{top}(\phi)$ there exists a compact and invariant subset $\Gamma_c\subset \Lambda$ such that:
\begin{enumerate}
\item The flow $\phi|_{\Gamma_c}$ is strictly ergodic.
\item The topological entropy $h_{top}(\phi|_{\Gamma})$ is equal to $c$.
\end{enumerate} 
Furthermore, considering $\mu_c$ as the unique invariant measure of $\phi_t|_{\Gamma_c}$, we conclude that $\phi$ has entropy flexibility.

\end{mainthm}

\begin{remark}
    To avoid unnecessary repetitions, we would like to remark that for all the results in this work, the entropy flexibility is achieved by constructing strictly ergodic subsystems.

\end{remark}

We will see that, by using finite extensions, the above theorem holds for any expansive system with the shadowing property instead of the sub-shift of finite type. Theorem~\ref{thm: SFT} has several useful applications. As we will see, it can be used to prove
flexibility for several classes of flows generated by differentiable vector fields over compact manifolds.
In our next result, we prove  Conjecture \ref{conj} in the low-dimensional scenario.

\begin{mainthm}
\label{thm: generic}
Entropy flexibility holds $C^1$-generically for:
\begin{enumerate}
    \item\label{thm:generic:1} Flows generated by vector fields over three-dimensional manifolds.
    \item\label{thm:generic:2} Surface homeomorphisms. 
   \end{enumerate}
\end{mainthm}

In the higher-dimensional setting, we can apply Theorem~\ref{thm: SFT} to investigate the entropy flexibility for a large set of flows exhibiting hyperbolic-like behavior. Hyperbolicity is a cornerstone of dynamical systems theory, playing a central role in studying stability. Introduced by Smale in~\cite{Smale}, it serves as both a source of chaotic dynamics and a foundation for structural stability. In the context of flows, however, hyperbolicity can only occur when singularities are isolated from regular orbits. Despite this restriction, several fundamental examples---most notably the Lorenz attractor---exhibit robust or persistent chaotic behavior while featuring singularities accumulated by regular orbits. To account for such phenomena, generalized notions of hyperbolicity have been developed, including sectional hyperbolicity (more generally, asymptotically sectional-hyperbolicity), multisingular hyperbolicity, and the star property (see~\cites{Ma,Liao,MPP1,MPP2,MS,BDL,SMV}).
Our next result shows that entropy flexibility holds for many systems with hyperbolic-like behavior.  
\begin{mainthm}\label{thm: variousflows}
\label{c.app}
Let $X$ be a $C^1$-vector field over $M$ and $\Lambda\subset M$ be a compact and invariant set. 
\begin{enumerate}
\item\label{thmC:a} If $X$ generates  a star vector flow, then $X$  has entropy flexibility.
\item\label{thmC:b} If $\Lambda$ is an asymptotically sectional-hyperbolic attracting set for $X$, then $X|_{\Lambda}$ has entropy flexibility. 
\end{enumerate}

\end{mainthm}

The diagram in Figure~\ref{fig:hyper} presents some well-known implications between several hyperbolic-like properties for vector field.

\begin{figure}[!ht]
\begin{center}
\scalebox{0.75}{ 
\begin{tikzpicture}[
  node distance=1.6cm and 2.8cm,
  box/.style={rectangle, draw, rounded corners, minimum width=3cm, minimum height=0.9cm, align=center, font=\footnotesize},
  imply/.style={->, double equal sign distance, thick}
  ]

  \node[box] (H) {Hyperbolic};
  \node[box, right=of H] (SH) {Sectional-Hyperbolic};
  \node[box, right=of SH] (MSH) {Multi-Singular Hyperbolic};

  \node[box, below=of SH] (ASH) {Asymptotically Sectional-Hyperbolic};

  \node[box, below=of MSH] (LS) {Star};

  \draw[imply] (H) -- (SH);
  \draw[imply] (SH) -- (ASH);
  \draw[imply] (SH) -- (MSH);
  \draw[imply] (MSH) -- (LS);

\end{tikzpicture}
}
\end{center}
    \caption{Relations between selected hyperbolic-like properties of vector field.}
    \label{fig:hyper}
\end{figure}
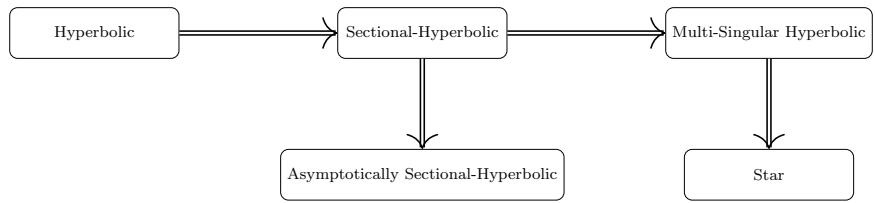

\newpage
\begin{remark}
Notice that:
\begin{enumerate}
\item The same holds if the manifold $M$ has a boundary and the vector field points inward at the boundary. We only need to consider the maximal invariant subset $M(X)=\bigcap_{t\geq 0}\phi_t(M)$ in place of the whole manifold such as the sectional-Anosov flows introduced in \cite{Mo}. In particular, the result holds for attracting sets. 
\item The asymptotically sectional-hyperbolic sets includes the Rovella attractor and the contracting horseshoes as archetypal examples (see \cites{SMV, MSM}).
\end{enumerate}
\end{remark}
Following the implications on Figure~\ref{fig:hyper} and the previous remarks, we obtain the following results as an immediate corollary.

\begin{mainclly}
The entropy flexibility is satisfied  for any of the vector fields or invariant sets  in the following list:

\begin{enumerate}
   
    \item Hyperbolic sets and more generally Axiom A or Anosov flows.
    \item Lorenz-like attractors and more generally sectional-hyperbolic or multi-singular hyperbolic flows
    \item Sectional-Anosov flows
    \item Rovella-like attractors.
\end{enumerate}
\end{mainclly}

As we will see, our results further apply to explore the entropy for the diffeomorphism scenario.  The starting point of our approach is the following immediate corollary of Theorem \ref{thm: SFT}, which is obtained by considering suspensions with constant roof function $\rho\equiv 1$.
\begin{corollary}\label{cor: SFT}
    If $\sigma:\Si\to\Si$ is a shift of finite type, then $\si$ has entropy flexibility.
\end{corollary}
With this result in hand, we shall see in Section \ref{sec: PH} that entropy flexibility is verified for some classes of partially hyperbolic diffeomorphisms. 

This article is organized as follows in Section \ref{section: prelim} we provide the reader with some basic concepts and results that will be used throughout this work. In section \ref{sec: SFT} we prove Theorem \ref{thm: SFT}. Section \ref{sec: smoothflows} is devoted to the study of entropy flexibility for smooth flows. In particular, we prove Theorems \ref{thm: generic} and \ref{thm: variousflows}. In Section \ref{sec: PH} we apply our findings to explore the entropy flexibility of partially hyperbolic diffeomorphisms.

\section{Preliminaries}\label{section: prelim}
In this section, we collect basic definitions used throughout this work and also some basic, well-known results. Hereafter, $M$ denotes a compact metric space which will also be regarded as a measurable space endowed with the Borel $\sigma$-algebra.  

\subsection{Discrete-time dynamics.}
Throughout this text, we will deal with invertible dynamics. By a discrete-time system, we mean a homeomorphism $ f: M \to M $. Later, when it becomes necessary, $ M $ will be regarded as a Riemannian manifold and $ f $ as a $ C^1 $-diffeomorphism. Denote by $\SK(M)$ the set of compact nonempty subsets of $M$ endowed with the Hausdorff metric: $$d^H(A,B)=\inf\{\eps>0; A\subset B_{\eps}(B) \textrm{ and } B\subset B_{\eps}(A)\}.$$ 
It turns out that $\SK(M)$ is a compact metric space, due the compactness of $M$.
A subset $ K \subset M $ is said to be {\it invariant} if $ f(K) = K $. In this context, if $ K \subset M $ is invariant, then $ f|_K $ can be seen as a dynamical system, which is simply a subsystem of $ f $. In this way, all the following concepts can also be stated for subsystems of $ f $. 

The {\it orbit} of a point $ x \in M $ is the set
$$\{ f^n(x); n \in \mathbb{Z} \}.$$
This defines the set of all iterates of $ x $ under the dynamics of $ f $. We say that $ f $ is {\it topologically transitive} if for any pair of non-empty open sets $ U $ and $ V $, there exists $ n > 0 $ such that $ f^{-n}(U) \cap V \neq \emptyset $.  It is well known that on compact metric spaces without isolated points, topological transitivity is equivalent to the existence of a {\it transitive point}, i.e., a point whose orbit is dense in $ M $. 
We say that a compact and invariant subset $ K \subset M $ is a {\it minimal set} if it does not contain any non-trivial compact and invariant subsets. Also, we say that $ f $ is {\it minimal} if $ M $ itself is a minimal set. It is classical that $ f $ is minimal if and only if every point in $ M $ is dense in $ M $.   
Let $ f\colon M \to M $ and $ g\colon K \to K $ be two homeomorphisms. We say that $ f $ \textit{factors} over $ g $ if there is a continuous surjective map $ \pi: M \to K $ (a \textit{factor map}) such that $ \pi \circ f = g \circ \pi $. If in addition $\pi$ is a homeomorphism, we say that $f$ and $g$ are topologically conjugated.

Now, let us recall the definition of topological entropy. Let $ K $  be a compact subset of $ M $ (not necessarily invariant). Fix an integer $n > 0 $ and $\varepsilon>0$. A subset $ Q \subset K $ is called an {\it $ n $-$ \varepsilon $-separated subset} of $ K $ if for any pair of distinct points $ x, y \in K $, there exists $ 0 \leq n_0 \leq n $ such that
$$ d(f^{n_0}(x), f^{n_0}(y)) > \varepsilon. $$ 
This condition ensures that points in the subset $ Q $ are sufficiently far apart after applying the map $ f $ iteratively. Let $ S(n, \varepsilon, K) $ denote the maximal cardinality of an $ n $-$ \varepsilon $-separated subset of $ K $. Notice that $ S(n, \varepsilon, K) $ is always finite due to the compactness of $ M $.

\begin{definition}
We define the topological entropy of $ f $ on $ K $ as the number
$$ h(f, K) = \lim\limits_{\varepsilon \to 0} \limsup\limits_{n \to \infty} \frac{1}{n} \log(S(n, \varepsilon, K)). $$
The topological entropy of $ f $ is defined as the number
$$ h(f) = h(f, M). $$ 
\end{definition}

Next, we recall the concept of the shadowing property. A sequence $ (x_n)_{n \in \mathbb{Z}} $ is called a {\it $ \delta $-pseudo orbit} if $ d(f(x_n), x_{n+1}) \leq \delta $ for every $ n \in \mathbb{Z} $.  We say that a $ \delta $-pseudo-orbit $ (x_n) $ is {\it $ \epsilon $-shadowed} if there exists $ y \in M $ such that $ d(f^n(y), x_n) \leq \epsilon $ for every $ n \in \mathbb{Z} $. 

\begin{definition}
We say that $ f $ has the {\it shadowing property} if, for every $ \epsilon > 0 $, there exists $ \delta > 0 $ such that every $ \delta $-pseudo-orbit is $ \epsilon $-shadowed by some point in $ M $. 
\end{definition}

\subsection{Smooth ergodic theory of discrete-time systems}
We say that a Borel probability measure $ \mu $ on $ M $ is an \textit{invariant measure} for $ f $ if $ \mu(f^{-1}(A)) = \mu(A) $ for every Borel subset $ A $. Denote by $ \mathcal{M}_f(M) $ the set of invariant measures for $ f $. The Krylov-Bogolyubov Theorem states that $ \mathcal{M}_f(M) $ is always nonempty. An $ f $-invariant measure is \textit{ergodic} if $ \mu(A) \mu(A^c) = 0 $, for every Borel and $ f $-invariant set $ A $. Let us denote $ \SM^e_f(M) $ for the set of ergodic measures, which is also always nonempty. 
The system $ f $ is considered uniquely ergodic if $ f $ admits a unique invariant measure.

\begin{definition}
We say that $ f $ is strictly ergodic if it is minimal and uniquely ergodic. This means that there is exactly one invariant measure, and the system is minimal in the sense that there are no proper invariant subsets.
\end{definition}

For each $x\in M$, define the $n$-th empirical measure of $x$ as 
$$\SE_n(x)=\frac{1}{x}\sum_{j=0}^{n-1}\delta_{f^j(x)},$$
where $\delta_y$ denotes the Dirac measure at $y$.

The space $\SM(M)$ can be endowed with a metric order to make it a compact metric space as follows (see \cite{Du}). For any bounded real-valued function $\phi:M\to \R$, denote $$||\phi||_L=\sup\left\{\frac{|\phi(x)-\phi(y)|}{d(x,y)};x\neq y      \right\}.$$
Let $BL(M)$ denote the set of bounded real-valued Lipschitz functions on $M$, i.e.,  satisfying $||\phi||_L<\infty$. The set $BL(M)$ is dense in the set $C(M,\R)$ consisting of continuous real-valued functions on $M$. Fix $||\phi||_{BL}=||\phi||_{\infty}+ ||\phi||_L$. Let $(\phi_n)$ be a sequence of functions in $BL(M)$ satisfying $||\phi_n||_{BL}\leq 1$ and dense in $C(M,\R)$. So, for any pair of measures $\mu,\nu\in \SM(M)$, define $$\mathfrak{d}(\mu,\nu)=\sum_{n=1}^{\infty}\frac{1}{2^n}\left|\int \phi_nd\mu-\int \phi_nd\nu \right|.$$ 
In this way, $\mathfrak{d}$ is a metric on $\SM(M)$ whose induced topology coincides with its weak$^*$-topology. So, whenever we write $\mu_n\to \mu$, we mean that the convergence is being considered in the metric $\mathfrak{d}$.
We say that $x$ is a {\it generic point of $\mu\in \SM_f(M)$}, if $\SE_n(x)\to \mu.$

The \textit{metric entropy} of an $ f $-invariant measure $ \mu $ is given by
\begin{displaymath}
h_{\mu}(f) = \sup \lbrace h_{\mu}(f, \mathcal{P}) : \mathcal{P} \text{ is a finite and measurable partition of } M \rbrace, 
\end{displaymath}
where
\begin{displaymath}
h_{\mu}(f, \mathcal{P}) = -\lim_{n \to \infty} \frac{1}{n} \sum_{P \in \mathcal{P}_n} \mu(P) \log \mu(P), \quad
 \mathcal{P}_n = \mathcal{P} \vee f^{-1}(\mathcal{P}) \vee \cdots \vee f^{n-1}(\mathcal{P}).
\end{displaymath}

The celebrated variational principle relates the metric entropies of a continuous map and its topological entropy:

\begin{theorem}[Variational Principle]
For any continuous map $ f $, it holds:
\begin{eqnarray*} 
h(f) &=& \sup \{ h_{\mu}(f) ; \mu \in \mathcal{M}_f(M) \}\\ 
&=& \sup \{ h_{\mu}(f) ; \mu \in \mathcal{M}^e_f(M) \}.
\end{eqnarray*}
\end{theorem}

We now turn our attention to the differentiable ergodic theory of diffeomorphisms. Let $\mu$ be an $f$-invariant Borel probability measure.

\begin{definition}
A number $\lambda \in \mathbb{R}$ is called a \emph{Lyapunov exponent} of $f$ at $x$ (with respect to $\mu$) if there exists a nonzero vector $v \in T_x M$ such that
$$
\lim_{n \to \infty} \frac{1}{n} \log \Vert Df^n(x) v \Vert = \lambda.
$$
\end{definition}

\begin{theorem}[Oseledets' Theorem]
Let $f: M \to M$ be a $C^1$ diffeomorphism and let $\mu$ be an $f$-invariant Borel probability measure. Then there exists a $\mu$-full measure set $R \subset M$ such that for every $x \in R$, there exist:
\begin{itemize}
    \item real numbers $\lambda_1(x) \leq \cdots \leq \lambda_k(x)$;
    \item a measurable $Df$-invariant splitting $T_x M = E^1_x \oplus \cdots \oplus E^k_x$,
\end{itemize}
such that for all $v \in E^i_x \setminus \{0\}$,
$$
\lim_{n \to \pm\infty} \frac{1}{n} \log \Vert Df^n(x) v \Vert = \lambda_i(x).
$$
\end{theorem}

\medskip

When all Lyapunov exponents of a measure are nonzero, the measure exhibits strong dynamical behavior, often leading to the existence of stable and unstable manifolds.

\begin{definition}
We say that an $f$-invariant ergodic measure $\mu$ is \emph{hyperbolic} if all its Lyapunov exponents are nonzero $\mu$-almost everywhere.
\end{definition}

\subsection{Basic Symbolic Dynamics}
Let us begin by recalling some important facts about symbolic dynamics. The reader interested in a more detailed treatment of this topic is referred to \cite{Bruin} or \cite{Kurka} for a more detailed exposition on the subject. Let $ A_n = \{ 0, 1, \dots, n-1 \} $ and denote $ \Sigma_n = A_n^{\mathbb{Z}} $, i.e. $ \Sigma_n $ is the set of bi-infinite sequences formed by elements of $ A_n $. The set $ \Sigma_n $ is endowed with the metric ($s,s'\in \Sigma_n$, $s\neq s'$):
$$
d(s, s') = \frac{1}{2^{|i|}},
$$
where $i$ is an integer with largest $|i|$ and such $ s_j = s'_j $, for every $ |j| < |i| $, but $ s_i \neq s'_i $. It is easy to see that $(\Sigma_n,d) $ is a compact metric space. We define the \textit{shift map} 
$$
\sigma: \Sigma_n \to \Sigma_n,
$$
defined by $ \sigma((s_i)) = (s_{i+1}) $. The map $\sigma$ is clearly a homeomorphism. We call $(\Sigma_n,\sigma)$ the \textit{full shift on $n$ symbols}.
It is well known that $ h(\sigma) = \log(n) $. If  $\Si\subset\Si_n$ is  compact and $\si$-invariant, we say that $(\Sigma,\si|_\Sigma)$ is a subshift of $(\Si_n,\sigma)$.  We often simplify the notation and say that $\Si$ is a shift space or subshift. An $n\times n$-matrix with entries in $\{0,1\}$ is called a {\it transition matrix}. 
\begin{definition}
    We say that $\Si_A\subset \Sigma_n$ is a 
    \textit{ vertex shift}, if there is a transition matrix $A$ such that 
    $$
    \Si_A=\{(x_i)_{i\in \Z}\in \{0,...,k-1\}; A_{x_1x_{i+1}}={ 1}, \forall i\in \Z\}.
    $$
A subshift $(\Sigma,\sigma)$ is a \textit{shift of finite type (SFT for short)} if it is conjugated to a vertex shift $(\Sigma_A,\sigma)$.
\end{definition}
In simpler terms, $\Sigma$ is a shift of finite type if there is  $m$ and a finite collection of rules that determine whether sequences of symbols $x_0x_1\ldots x_{m-1}$ and $x_1\ldots x_{m-1}x_m$ can appear that overlap in $(x_i)_{i\in \Z}\in\Sigma$. An important fact about shifts of finite type is that they characterize the symbolic dynamics displaying the shadowing property, precisely:

\begin{theorem}
    A shift map is of finite type if and only if it has the shadowing property. 
\end{theorem}

Let $x\in \{0,...,k-1\}$ and for $-\infty\leq i\leq j \leq \infty$, define $x_{[i,j]}=x_ix_{i+1}\cdots x_j$. We say that $w$ is a word in $x$, if $w=x_{[i,j]}$, for some pair $i$ and $j$. The length of a word $w=w_0w_1\ldots w_{n-1}$ is defined as $|w|=n$. Given two words $w$ and $w'$, we say that $w'$ is a sub-word of $w$, if $w=x_{[i,j]}$ and $w'=x_{[i',j']}$, for some $i\leq i'\leq j'\leq j$.  

Let $\Si$ be a shift space. The language of $\Si$ is defined to be the set of words 
$$
\SL(\Si)=\{w: w \textrm{ is a finite sub-word of a point } x\in \Si \}.
$$
For any $n\geq 1$ we define $\SL_n(\Si)=\{w\in \SL(\Si): |w|=m\}$ formed by words of length $n$. The next two concepts will play a central role in our proofs:
\begin{definition}
    We say that a shift space $\Si$ is a renewal shift if there is a set of words $\SW$ so that $\Si$ is the space of sequences formed by free concatenations of the words in $\SW$.  
\end{definition}

\begin{definition}
    A renewal system $\Si$  over the set of words $\SW$ is said to be \textit{uniquely decipherable} if for every $x\in \Si$, there is a unique way of writing $x$ as a concatenation of words in  $\SW$. 
\end{definition}

 In what follows, we will need the following useful property of renewal systems, proved first in \cite{Ai}.
\begin{theorem}\label{thm: Ai}
    Every uniquely decipherable renewal shift is a shift of finite type.  
\end{theorem}

\subsection{Continuous-Time Dynamical Systems}
Next, we recall some basic facts about flows.  The reader interested in a complete treatise about the concepts presented here, as well as an excellent exposition of the flows theory, is referred to \cite{Kat}. A continuous-time system or {\it flow} is a  continuous map 
$ \phi: \mathbb{R} \times M \to M $ satisfying:
    \begin{enumerate}
        \item $ \phi(0, x) = x $, for every $ x \in M $.
        \item $ \phi(t + s, x) = \phi(t, \phi(s, x)) $, for any $ x \in M $ and any $ t, s \in \mathbb{R} $.
    \end{enumerate}
    We denote by $ \phi_t $ the time-$ t $ map $ \phi(t, \cdot) $.  
When $ M $ is a compact Riemannian manifold, and $ \phi $ is a 
 $C^1 $-map, we can equivalently define $ \phi $ through its velocity vector field. Indeed, if $ X $ is a $ C^1 $-vector field over $ M $, then $ X $ induces on $ M $ a flow whose velocity vector field is $ X $. On the other hand, one can assign to any $ C^1 $-flow $\phi$ on $ M $, a velocity vector field  $ X_\phi $. In this work, we will often denote by  $ X $ the velocity vector field that induces $ \phi $, whenever there is no risk of confusion about the flow generated by $X$.

We say that a subset $ K \subset M $ is \textit{$ \phi $-invariant} if $ \phi_t(K) = K $, for every $ t \in \mathbb{R} $. A point $ x \in M $ is a \textit{singularity} for $ \phi $ if $ X(x) = 0 $. We denote the singularities of $ \phi $ by $ \Sing(
  X)$  or simply $\Sing(\phi)$. A point $ x \notin \Sing(X) $ is said to be a \textit{regular point}. A regular point is \textit{periodic} for $ \phi $ if there is $ \eta > 0 $ such that $ \phi_{\eta}(x) = x $. We say that a flow is a \textit{regular flow} if its set of singularities is empty.

Similarly to the homeomorphism setting, we say that a Borel probability measure $ \mu $ is $ \phi $-invariant,  or that $\phi$ preserves the measure $\mu$, if one has $ \mu(\phi_t(A)) = \mu(A) $, for every $ t \in \mathbb{R} $ and every measurable set $ A $. A $ \phi $-invariant measure is said to be ergodic if $ \mu(A) \mu(A^c) = 0 $, for every $ \phi $-invariant Borel set. We  denote $ \mathcal{M}_\phi(M) $ and $ \mathcal{M}_\phi^e(M) $ for the sets of invariant and ergodic measures, respectively.

\begin{definition}
    The \textit{topological entropy} of a flow $ \phi $ is the topological entropy of its time-one map, i.e., 
    $$
    h_{top}(\phi) := h_{top}(\phi_1).
    $$
    Similarly, if $\mu$ is an invariant measure for the flow $\phi$ then we define $h_\mu(\phi):=h_\mu(\phi_1).$
\end{definition}
Now, we shall recall the concept of suspension flow, which will play a central role in our exposition. Let $ f\colon M \to M $ be a homeomorphism and $ \rho\colon M \to \mathbb{R} $ be a continuous function. Let $ \overline{M} $ be the quotient space of $ M \times \mathbb{R} $, through the equivalence relation $ (x, \rho(x)) \sim (f(x), 0) $.  We call $\rho$ the \textit{roof function}.

\begin{definition}\label{Def-Susp}
    The \textit{suspension flow} of $ f $ with the roof function $ \rho $ is the flow $ \phi\colon \mathbb{R} \times \overline{M} \to \overline{M} $ defined by setting $ \phi(t, (x, s)) = (f^n(x), s') $, where $ n $ and $ s' $ satisfy
    $$
    \sum_{j=1}^{n-1} \rho(f^j(x)) + s' = t + s, \quad 0 \leq s' \leq \rho(f^n(x)).
    $$
\end{definition}
 It is not hard to see that $\phi$ form Definition~\ref{Def-Susp} is indeed a flow. An interesting fact about suspension flows is that there is a bijective map $ \Phi: \mathcal{M}_f(M) \to \mathcal{M}_\phi(
 \overline{M}) $ mapping each invariant measure $ \mu $ of $ f $ into an invariant measure $ \Phi(\mu) = \hat{\mu} $ of  $\phi$ on  $\overline{M}$. The measure $ \hat{\mu} $ is called the lift measure of $ \mu $. Moreover, one can relate the entropies of $ \mu $ and $ \hat{\mu} $ through the celebrated Abramov formula:
$$
h_{\hat{\mu}}(\phi) = \frac{h_\mu(\sigma)}{\int \rho \, d\mu}.
$$

The following definition is a standard generalization of the concept of Smale's horseshoe to the setting of flows.  It is a nice tool for estimating topological entropy.

\begin{definition}
    Let $\Lambda\subset M$ be a compact and invariant set for a flow $\phi$ on $M$. We say that $\Lambda$ is
    a \textit{horseshoe} of $\phi$ if there is a
suspension flow $\psi$ on $\overline{\Sigma_n}$, $n\geq 2$ with a continuous roof function $\rho\colon \Sigma_n
\to (0,\infty)$ and a homeomorphism $h\colon \Lambda\to \overline{\Sigma_n}$ conjugating these flows, i.e.
$\psi_t\circ h=h\circ \phi_t$ for every $t\in \R$.
\end{definition}

\section{Entropy Flexibility and Suspension Flows}\label{sec: SFT}
In this section, we start proving our main results. We begin providing the reader with the proof of Theorem \ref{thm: SFT}. Specifically, we shall see that any suspension of a shift of finite type has entropy flexibility. We will first need to recall some preliminary concepts and results to achieve our goal.

\subsection{Preliminary Results}
The following two results are technical lemmas that will help us construct measures that satisfy reasonable estimates.

\begin{thm}[\cite{LiOp}]\label{thm:LO}
Suppose that $f$  is transitive and has the shadowing
property. Then for every $f$-invariant measure $\mu$ on  and every $0\leq c\leq h_{\mu}(t)$, there exists a sequence of ergodic measures $\mu_n$ supported on almost $1$-$1$ extensions
of odometers such that $\lim\limits_{n\to\infty} \mu_n = \mu$ and $\lim\limits_{n\to \infty} h_{\mu_n}(f)=c$.
\end{thm}

\begin{lemma}[\cite{LiOp}]\label{lem:measure-approx}
Let $(M,d)$ be a compact metric space and $\eps>0$.
\begin{enumerate}
  \item \label{enum:measure-approx-1}
  For any sequence $(x_i)_{i=0}^\infty$ of points in $M$ and any nonempty finite subsets $A,B$ of $\N_0$, we have
$$
   \mathfrak{d}\Bigl( \frac{1}{|A|}\sum_{i\in A} \delta_{x_i}, \frac{1}{|B|}\sum_{i\in B}\delta_{x_i}\Bigr)
  \leq \frac{|A|+|B|}{|A|\cdot|B|}|A\Delta B| + \frac{\bigl||A|-|B|\bigr|}{|A|\cdot|B|}|A\cap B|.
$$
  \item \label{enum:measure-approx-2}
  For two sequences $(x_i)_{i=0}^{m-1}$ and $(y_i)_{i=0}^{m-1}$ of points in $M$,
    if $d(x_i,y_i)<\eps$ for $i=0,1,\dotsc,m-1$, then
$$
  \mathfrak{d}\Bigl( \frac{1}{m}\sum_{i=0}^{m-1} \delta_{x_i},
        \frac{1}{m}\sum_{i=0}^{m-1} \delta_{y_i}\Bigr)<\eps.
$$
 \item\label{enum:measure-approx-3}
 If $\mu_i,\mu\in \SM(M)$ are such that $\mathfrak{d}(\mu_i,\mu)<\eps$ for $i=1,\ldots, k$,
 then for any choice of $\alpha_i\in [0,1]$
 with $\sum_{i=1}^k\alpha_i=1$ we have
$$
     \mathfrak{d}\Bigl(\sum_{i=1}^K \alpha_i \mu_i,\mu\Bigr)<\eps.
$$
\end{enumerate}
\end{lemma}

\begin{lemma}[\cite{LiOp}]\label{lemma2}
    Let $\mu\in \SM^e_f(M)$. For every $\kappa>0$, there is $\eps>0$ such that for every neighborhood $\SU\subset \SM(M)$ of $\mu$, there is $N>0$ such that if $n>N$, there is an $n$-$\eps$-separated set $\Gamma_n\subset  \su(\mu)$ satisfying:
\begin{enumerate}
    \item  $\SE_n(x)\subset \SU$, for every $x\in \Gamma_n$.
    \item If $h_{\mu}(f)<\infty$, then 
    $$
    \left|\frac{1}{n}\log(\#\Gamma_n)-h_{\mu}(f)   \right|<\kappa.
    $$
    \item  If $h_{\mu}(f)=\infty$, then 
    $$
    \frac{1}{n}\log(\#\Gamma_n)> \frac{1}{\kappa}.
    $$
\end{enumerate}
\end{lemma}

In addition to the previous lemmas, we will need the following result, which is a consequence of \cite{Tal}*{Lemma 4.1.2}.

\begin{lemma}\label{thm:specialword}
    Let $\Si$ be a sub-shift of topological entropy $h>0$. Then 
     there exists $L$ such that for every $l\geq L$   
    there exists  
    $w^l=w_0w_1...w_l\in \SL(\Sigma)$ which satisfies
    $$
    \max (\{0< k<l : w_{l-k} ...w_{l-1} = w_0 ...w_{k-1}\} \cup \{0\})<l/4.
    $$
\end{lemma}

\subsection{Proof of Theorem \ref{thm: SFT}} Now we are in position to prove Theorem \ref{thm: SFT}. Hereafter, let   
$ (\Si,\si)$ denote a shift of finite type and let $\rho:\Si\to \R^+$ a roof function. Let $\phi$ be the suspesion flow of $\si$ with roof $\rho$. Recall that  
there is a bijection  $\Psi:  \SM_\sigma(\Si)\to \SM_\phi(
 \overline{\Si})$ assigning to any invariat measure $\mu$ of $\si$, an invariant measure $\Psi(\mu)=\hat\mu$ for $\psi$. In addition, this bejection also relates the entropies of $\mu$ and $\hat \mu$ through the Abramov's formula: $$h_{\hat \mu}(\phi)=\frac{h_\mu(\sigma)}{\int\rho d\mu}.$$
With all these observations, the proof of Theorem 
\ref{thm: SFT} is reduced to prove the following Theorem:

\begin{thm}\label{thm: reduction}
Let $\mu$ be any invariant measure on $\Sigma$ with $h_\mu(\sigma)>0$,  let $\rho\colon \Si\to \R^+$ be a continuous function  and denote $h^*=h_\mu(\sigma)/\int \rho d\mu$. For every $c\in (0,h^*)$ and every $\eps>0$ there exists $\nu\in M_\sigma(\Sigma)$ such that:
\begin{enumerate}
    \item $h_\nu(\sigma)= c \int \rho d\nu$,
    \item  $(Y,\sigma)$ is strictly ergodic, where $Y=\supp \nu$. 
\end{enumerate}
\end{thm}
Let us first assume that Theorem \ref{thm: reduction} holds and use it to prove Theorem \ref{thm: SFT}
\begin{proof}[Proof of Theorem \ref{thm: SFT}]
Let $\phi:\R\times 
 \overline{\Si}\to 
 \overline{\Si}$  be a suspension flow over a shift of finite type  $(\Si,\si)$ 
 and let $\rho\colon 
\Si\to \R_+$ be the roof function defining $\phi$. 
By the variational principle, there is a sequence of measures $\hat\mu_n$ so that $h_{\mu_n}(\phi)\to h_{top}(\phi)$. Fix such a measure $\hat\mu_n$ and let $\mu_n=\Psi^{-1}(\hat\mu)$. Fix $c\in [0,\frac{h_{\mu_n}(\si)}{\int\rho d\mu_n})$.  By Theorem \ref{thm: reduction}, there is a $\si$-invariant  measure $\nu$ so that $\nu$ is a measure of maximal entropy on $X=\supp(\nu)$ and $h_{\nu}(\si)=c\int\rho d\nu$,  $X$ is minimal and $(X,\sigma)$ is strictly ergodic. By Abramov's formula we have $$h_{\hat\nu}(\phi)=\frac{h_\nu(\si
)}{\int\rho d\nu}=c.$$ 
Moreover, since 
 $(X,\si)$ is uniquely ergodic,
we get $h_{top}(\phi|_{\hat X})=h_{\hat\nu}(\phi)=c$ and the proof is finished. 
\end{proof}

The proof of Theorem \ref{thm: reduction} is quite technical and long. For this reason, we decided to break the proof into parts and extract technical results from the proof that will make the reading easier and also, can be of independent interest for further applications. We begin with the following result, which will give us the proof Theorem \ref{thm: reduction} as a corollary.

\begin{thm}\label{thm: reduction2}
Let $\si$ be a transitive shift of finite type, let $\rho\colon \Si\to \R^+$ be a continuous function and $\mu$ an invariant measure for $\si$ with $h_\mu(\si)>1$.  Denote $h^*=h_\mu(\sigma)/\int \rho d\mu$. For every $c\in [0,h^*)$,  there exist 
 sequences of positive numbers $\de_1>\de_2>\de_3>\cdots$,   $\ka_1>\ka_2>\ka_3>\cdots$ decreasing to zero and a sequence $\Si=X_0\supset X_1\supset X_2\supset X_3\supset\cdots$ of subshifts of finite type  such that, for every $n\geq 0$, it holds:

$$
(1-\delta_{n+1})\int \rho d\nu_n < \int \rho d\eta < (1+\delta_{n+1})\int \rho d\nu_n 
$$
 and $\mathfrak{d}(\eta,\nu_n)\leq 2\ka_n$ for every measure $\eta\in \SM_\si(X_n)$ and

$$
(1+\delta_{n+1})^2 c\int \rho d\nu_n \leq 
h_{top}(X_{n+1})
\leq
(1+3\delta_{n+1})c\int \rho d\nu_n,
$$
where $\nu_0=\mu$, and $\nu_n$ is  defined as the measure of maximal entropy of $X_n$ 
for every $n\geq 1$.

Furthermore, there exists an increasing sequence $s_n$ such that $\mathcal{L}_{s_n}(X_i)=\mathcal{L}_{s_n}(X_n)$ for every $i\geq n$
and for every words $w\in \mathcal{L}_{s_n}(X_n)$, $v\in \mathcal{L}_{s_j}(X_n)$, $j<n$, $v$ is a subword of $w$.
\end{thm}  

Next we are going to prove Therem \ref{thm: reduction} by assuming Theorem \ref{thm: reduction2} is valid (we will prove it later).

\begin{proof}[Proof of Theorem \ref{thm: reduction}]
 Apply Theorem~\ref{thm: reduction2} obtaining 
sequences $X_n$ and $\delta_n$
satisfying conditions as in the statement and denote $$X=\bigcap_{n=0}^{\infty}X_n.$$

 By definition, for every $\eta,\nu \in \SM_\si(X)\subset \SM_\si(X_n)$ we have $\mathfrak{d}(\nu,\mu)<4\kappa_n\to 0$, hence $(X,\sigma)$ is uniquely ergodic. By the ``furthermore'' part in Theorem~\ref{thm: reduction2}, for every $n\geq 1$, for every words $w\in \mathcal{L}_{s_n}(X)$, $v\in \mathcal{L}_{s_j}(X)$, $j<n$, $v$ is a subword of $w$
proving that $(X,\sigma)$ is minimal. This proves that it is strictly ergodic.

Let $\nu$ be  
 the unique measure in $\SM_\si(X)$
 Since $\nu\in \SM_\si(X_n)$ for every $n$, we have:
$$
\frac{h_\nu(\sigma)}{\int \rho d\nu}\leq \frac{h_{top}(X_{n})}{(1-\delta_n)\int \rho d \nu_{n-1}}\leq c\frac{1+3\delta_n}{1-\delta_n}\longrightarrow c.
$$

On the other hand,  since the function $\eta\mapsto h_\eta(\sigma)$ is upper semicontinuous on $\SM_\si(\Sigma)$ and $\lim_k \nu_k=\nu$, we have
$$
h_\nu(\sigma)=h_{top}(X) {\geq}\limsup_{k} h_{\nu_k}(\sigma)=
\limsup_{k}h_{top}(X_k)\geq \limsup_k c\int \rho d\nu_k.
$$
Therefore for every $\gamma>0$ and infinitely many sufficiently large $k$ we have
$$
h_\nu(\sigma) \geq c\int \rho d\nu_k-\gamma.
$$
But again, since $\SM_{\sigma}(X)\subset \SM_{\sigma}(X_n)$, for every $n<0$, we obtain by Theorem \ref{thm: reduction2}:
$$
\frac{h_\nu(\sigma)}{\int \rho d\nu}\geq\frac{c}{(1+\delta_{k+1})}-\frac{\gamma}{\int \rho d\nu
}.
$$
By going with  $k\to \infty$ and then with $\gamma\to 0$, we  obtain:
$$
\frac{h_\nu(\sigma)}{\int \rho d\nu}=c.
$$
The proof is completed.
\end{proof}

\begin{proof}[Proof of Theorem \ref{thm: reduction2}]
Let 
$ (\Si,\si)$
be a transitive shift of finite type. Recall that every subshift is expansive and, being of finite type, it also has the shadowing property.  By expansiveness, every shift admits a measure of maximal entropy (e.g. see \cite{Walters}*{Remark (2), p. 192}). 
 Let $\rho \colon \Si \to \R_+$ be a continuous function.

Fix any invariant measure $\mu$ for $\si$ with positive entropy,  denote $h^*=h_\mu(\sigma)/\int \rho d\mu$ and fix any $c\in (0,h^*)$. By the choice of $c$, we have
$c\int \rho d\mu< 
h_{\mu}(\sigma)$ 
and therefore there is $\delta_1>0$  so that the following estimate holds:
          $$c\int\rho d\mu < (1+\delta_1)^2 c\int\rho d\mu<
(1+3\delta_1)c\int \rho d\mu<h_\mu(\sigma)$$
and  trivially we also have  
 $$(1-\delta_1)\int \rho d\mu < \int \rho d\mu < (1+\delta_1)\int \rho d\mu.
$$
Fix \begin{equation}\kappa_1\leq \frac{\int\rho d\mu}{4}\min\left\{ c|(1+\delta_1)^2-(1+3\delta_1)|,\frac{\delta_1}{2}\right\}.\end{equation}
Next, consider $$c_1=\frac{c\int \rho d\mu |(1+\delta_1)^2+(1+3\delta_1)|}{2}.$$ Since $\sigma$ is transitive and has the shadowing property, by Theorem \ref{thm:LO}  there exist a minimal subshift $Y_1\subset \Si$ and an ergodic measure of maximal entropy $\nu_1$ for $\sigma|_{Y_1}$ so that the following estimates hold:
\begin{equation}
     \mathfrak{d}(\nu_1,\mu)\leq \kappa_1  \textrm{ and } |h_{\nu_1}(\sigma)- c_1|\leq\kappa_1.
     \end{equation}
 If we consider any invariant measure $\eta$ such that 
\begin{equation}
     \mathfrak{d}(\nu_1,\eta)\leq \kappa_1  \textrm{ and } |h_{\nu_1}(\sigma)- h_\eta(\sigma)|\leq\kappa_1,
     \end{equation}
then we obtain that

     \begin{equation}\label{int:est:eta}
          (1+\delta_1)^2 c\int \rho d\mu \leq h_{ \eta}(\sigma)=h_{top}(Y_1)
\leq
(1+3\delta_1)c\int \rho d\mu
\end{equation}
and     
 \begin{equation}(1-\delta_1)\int \rho d\mu < \int \rho d{ \eta} < (1+\delta_1)\int \rho d\mu.
\end{equation}
Again, Theorem \ref{thm:LO} provides us with a minimal subshift $Z_1\subset \Si$ positive entropy and disjoint from $Y_1$. Observe that since $Y_1$ and $Z_1$ are disjoint minimal sets, then  there is $K_1>0$ so that $\SL_k(Y_1)\cap\SL_k(Z_1)=\emptyset$, for any $k\geq K_1$. 
 Let 
$$0<\eps<\frac{\min\{\kappa_1,d(Y_1,Z_1)\}}{4},$$
 and $N>K_1$ be provided
by Lemma \ref{lemma2}, with respect to $\kappa_1$  and neighborhood 
$$\{\eta : \mathfrak{d}(\nu_1,\eta)<\kappa_1/2\}$$ 
of $\nu_1$. Let $0<4\delta<\eps$ be given by the shadowing property of  $\sigma$ with respect to $\frac{\eps}{4}$.

Now, we can choose open covers $\SU=\{U_1,...,U_{i_1}\}$ and $\SV=\{V_1,...,V_{j_1}\}$  of $Y_1$ and $Z_1$, respectively, by pairwise disjoint cylinders of diameter smaller than $\delta$. Since $\sigma$ is transitive,  for each pair $1\leq i\leq i_1$ and $1\leq j\leq j_1$ of indexes there are points $x_{i,j}\in U_i$, $y_{i,j}\in V_j$, and integers $n^+_{i,j},n^-_{j,i}>0$ satisfying  $$\sigma^{n^+_{i,j}}(x_{i,j})\in V_j \textrm{ and } \sigma^{n^-_{i,j}}(y_{i,j})\in U_i.$$   
Fix $$M=\max_{i,j}\{n^+_{i,j},n^-_{i,j}\}.$$

 Fix any $l>4(M+K_1)$
and let $w\in Z_1$ be  a word without ``large self-overlaps'' provided by Lemma~\ref{thm:specialword}. 
Let  $V^-=V_{j}$ and $V^+=V_{j'}$ denote the elements of $\SV$ containing $w$ and $\si^l{w}$, respectively.

For $n>N$, let $\Gamma'_n\subset Y_1$ be the $n$-$\eps$-separated set given by Lemma \ref{lemma2}.
 In what follows, we will enlarge $N$ consecutively. First, 
we assume that $N>N_1$ below. 
\begin{claim}\label{claim:1}
     There is $N_1>0$ such that for $n\geq N_1$ there are $\Ga_n\subset 
    \Ga_n'$ and $U_i,U_{i'}\in \SU$ such that:
    \begin{enumerate}[(i)]
        \item\label{c1:con1} For every $x\in  \Ga_n $ we have $x\in U_i$ and $\si^n(x)\in U_{i'}$.
        \item\label{c1:con2}  $\left|\frac{1}{n+ n^+_{i,j}+n^-_{j,i}+l}\log(\#\Ga_n)-h_{\nu_1}(\sigma)\right|\leq
         \ka_1$.
    \end{enumerate}
\end{claim}
\begin{proof}
    Fix $n\geq N$. For any each pair $1\leq i<i'\leq i_1$, define $$A_{i,i'}=\{x\in  \Ga'_n 
    : x\in U_i \textrm{ and } \si^n(x)\in U_{i'}\}.$$
    Since $\SU$ is a finite partition  of $Y_1$ 
    the family $\SA=\{A_{i,i'} : 1\leq i\leq i'\leq i_1\}$ forms a finite partition of  $\Ga'_n$. 
    Denote $A=\#\SA$ and observe that $A$ is independent of $n$. So,  there is at least one subset $A_{i,i'}$ satisfying $\#A_{i,i'}\geq \frac{\# \Ga'_n
    }{A}$. Take such a subset and denote $
    \Ga_n=A_{i,i'}$. Hence, the equality $$\lim\limits_{n\to \infty}\frac{1}{n}\log(\# \Ga_n
    )=\lim\limits_{n\to\infty}\frac{1}{n}\log(\frac{\# \Ga'_n
    }{A})=\lim\limits_{n\to \infty}\frac{1}{n}\log(\# \Ga'_n
    )$$
    ensures that for $n$ large enough, we have
    $$\left|\frac{1}{n}\log(\#
     \Ga'_n)-h_{\nu_1}(\sigma)\right|\leq \left|\frac{1}{n}\log(\#\Ga'_n)-\frac{1}{n}\log(\#\Ga_n)\right|+\left|\frac{1}{n}\log(\# \Ga_n)-h_{\nu_1}(\sigma)\right|\leq 
    \ka_1/2.$$
     But taking $n$ sufficiently large, we can make the following
    $$
    \left|\frac{1}{n}\log(\#\Ga_n)-\frac{1}{n+ n^+_{i,j}+n^-_{j,i}+l}\log(\#\Ga_n)\right|
    $$
    arbitrarily small, proving condition \eqref{c1:con2}.
    Moreover, by the construction of $\Ga_n$ condition \eqref{c1:con1} is immediately satisfied.
\end{proof}

Let $\Ga_n=\{x_1,....x_{s_n}\}$  be provided by Claim~\ref{claim:1} and let $U^-=U_i$ and $U^+=U_{i'}$,  where $n$ is sufficiently large, as will be specified later.
Now, we shall construct a family of $\delta$-pseudo orbits connecting the points in $\Gamma_n$ and $w$.  
First fix $u=x_{i',j}\in U^+$, $v=y_{i,j'}\in V^+$, $n_u=n^+_{i',j}$, $n_v=n^-_{i,j'}$ and recall that: $$\si^{n_u}(u)\in V^- \textrm{ and } \si^{n_v}(v)\in U^-.$$
For any $\alpha\in  \{1,...,s_n\}$ define the following sequences:
\begin{itemize}
    \item  $T_\alpha=x_\alpha,\si(x_{\alpha}),...,\si^{n-1}(x_\alpha)$.
    \item  $T_u=u,\si(u),...,\si^{n_u-1}( u )$. 
    \item  $T_{v}=v,\si(v),...,\si^{n_{v}-1}(v)$.
    \item $W=w,\si(w),...,\si^{l-1}(w)$.
\end{itemize}
Denote $S_{\alpha}=T_{v}T_\alpha T_{u}W$. For each $y\in \{1,...,s_n\}^{\Z}$, let $P_{y}$ be $\delta$-pseudo-orbit obtained  by concatenation of pseudo-orbits $S_i$ following
the assignment:
$$
P : \ldots y_{-1}y_0y_1\ldots \mapsto \ldots S_{y_{-1}}S_{y_0}S_{y_1}\ldots =: P_y.
$$
By the shadowing property, each $ y\in \{1,...,s_n\}^{\Z}$ generates a point $x_y$ which is $\eps$-tracing the pseudo-orbit $P_y$. The point $x_y$ is unique by the expansiveness. Define 

$$
\hat{X}_1=\{x_y : y\in \{1,...,s_n\}^{\Z}\}
$$
and

$$X_1=\bigcup_{j=0}^{n_v+n+n_u+l-1}\si^j(
\hat{X}_1
).$$
Notice that $X_1$ is closed and invariant, so $\sigma|_{X_1}$ is a sub-shift of $\si$.

Let us consider the words 
\begin{itemize}
\item $L_\alpha= (x_\alpha)_{[0,n)}$,
\item $L_u= u_{[0,n_u)}$,
\item $L_v= v_{[0,n_v)}$,
\item $L_w= \omega_{[0,l)}$.
\end{itemize}       
For each $\alpha\in \{1,...,s_n\}$, denote $W_{\alpha}=L_vL_\alpha L_uL_w$ and $\SW=\{
 W_\alpha;1\leq \alpha\leq s_n\}$.  Note that since $x_y$ is tracing $P_y$,
we must have  for every $m\in \Z$
$$
(x_y)_{ m}
=\sigma^{ m}
(x_y)_0=((P_y)_{ m}
)_0,
$$
hence  $X_1$ is the renewal system built by the  free concatenations of words in $\SW$. 
Clearly there is $k$ such that $|W_\alpha|=k$ for every 
 $W_\alpha \in \SW$.

\begin{claim} 
   The subshift $X_1$ is a shift of finite type. 
 \end{claim}
    \begin{proof}
To conclude that $X_1$ is a shift of finite type we shall apply Theorem \ref{thm: Ai}. For this reason, we need to show that every $x\in X_1$ is uniquely decipherable.    Recall that by the construction, there is $k$ such that $|W_\alpha|=k$, for every $1\leq \alpha\leq s_n$. But this implies that if $x\in  X_1$ then there is a sequence $(r_j)_{j\in \Z}\in \Z$ and an integer $i$ such that for every $j\in \Z$, we have $x_{[r_j,r_{j+1}]}\in \SW$, where $r_j=kj+i$.
We have to show that up to shift of 
 indexes, the sequence $r_j$ is unique, that is the reminder $0
\leq i<k$ is uniquely determined.

Suppose that there are two distinct sequences $r_j$ and $r'_j$ as above, i.e. they have different reminder modulo $k$. By the definition, there is $r>0$ so that $$0<|r_j-r'_j|=r<k.$$ 
 Shifting elements of the sequence $r_{j'}$ when necessary, we may assume that $r\leq k/2$.
We split the proof into two cases:

\begin{enumerate}[(C1)]
\item 
$r>M+K_1$. Take $\alpha,\alpha'$ such that $$x_{[i,k+i)}=W_{\alpha} \textrm{ and }  x_{[i+r-k,i+r)}=W_{\alpha'}.$$ 
Since  $n,l>K_1$, $r>M+K_1$  and $r\leq k/2$, we 
 obtain that $$(W_{\alpha'})_{[
 { k-r+M,k-r+M+K_1})}=(W_\alpha)_{[M,M+K_1)},$$
which implies that $\SL_{K_1}(Y_1)\cap\SL_{K_1}(Z_1)\neq \emptyset$
contradicting the choice of $K_1$.

\item 
$r\leq M+K_1$.
Take $\alpha,\alpha'$ such that $$ x_{[i,k+i)}=W_{\alpha} \textrm{ and }  x_{[i+r,k+i+r)}=W_{\alpha'}.$$
 In this cas,e have that
 $$
 w_{[r,l)}=(W_\alpha)_{[k-l+r,k)}=(W_\alpha')_{[k-l,k-r)}=w_{[0,l-r)}.
 $$
 But $l-r>l-M-K_1>l/4$ contradicting Lemma~\ref{thm:specialword}.

\end{enumerate}
 Indeed, $X_1$ is a uniquely decipherable renewal system, hence it is shift of finite type by Theorem \ref{thm: Ai}, proving the claim.
    \end{proof} 

We assume that $N>N_2$ below. 
\begin{claim}
    There is $N_2$ such that if $n$ in the construction of $X_1$ satisfies $n>N_2$ then $\mathfrak{d}(\nu,\nu_1)\leq 2\ka_1$, for every invariant measure $\nu\in \SM_{\si}(X_1)$.
\end{claim}
\begin{proof}
Fix any $x\in X_1$  and observe that $x=x_y$ for some $y\in \{1,...,s_n\}^{\Z}$. 
 Recall $k=n_v+n+n_u+l$, that is, $k$ is the common length of all words in $\SW$.
 Fix any $m>k$ and present it as $m=kq+r$, where $0\leq r<k$.
By the construction for each $i=0,..,q$ there are points $\{p_0,...,p_{q}\}\in \Ga_n$ so that for each $ik+n_v<j \leq ik+n_v+n$ we have $$d(\si^{j}(x),\si^{j-ik
-n_v}(p_j))\leq \eps$$
As a direct consequence of estimates in Lemma \ref{lem:measure-approx}, we obtain
$$ \mathfrak{d}\left(\SE_m(x),  \sum_{i=1}^q\frac{1}{n}\sum_{j=ik+n_v}^{ik+n_v+n}\de_{\sigma^{j}(x)} \right)\leq \frac{m+kq}{mkq}(m-kq)+\frac{m-kq}{mkq}(kq),$$
$$\mathfrak{d}\left(\sum_{i=1}^q\frac{1}{n}\sum_{j=ik+n_v}^{ik+n_v+n}\de_{\sigma^{j}(x)},  \sum_{i=1}^q\sum_{j=0}^{n-1}\de_{\sigma^{j}(p_i)} \right)\leq \eps,$$
and 
$$ \mathfrak{d}\left( \sum_{i=1}^q\sum_{j=0}^{n-1}\de_{\sigma^{j}(p_i)},\nu_1 \right)\leq{
 \frac{\ka_1}{2}}.$$
 
 But notice that $m-kq\leq (k+1)(n_u+nv)$. Hence
\begin{eqnarray*}&&    
\frac{m+kq}{mkq}(m-kq)+\frac{m-kq}{mkq}(kq)\\
&&\qquad\qquad\leq \frac{m+kq}{mkq}((k+1)(n_u+n_v))+\frac{(k+1)(n_u+n_v)}{mkq}(kq)\\
&&\qquad\qquad\leq \frac{3(k+1)(n_u+n_v)}{kq}\leq \frac{3(k+1)(n_u+n_v)}{n}\\
\end{eqnarray*} 

Since $k, n_u,n_v$ are bounded, there is $N_2$ such that if $n>N_2$, 
then the above estimates imply that
$$\left|\SE_m(x)-  \nu_1\right|\leq \eps+ \eps+  \frac{\ka_1}{2}\leq 
\ka_1.$$ 

Fix any $\nu\in \SM_\si(X_1)$. Since $(X_1,\sigma)$ is transitive and has shadowing property, then
every measure in $\SM_\si(X_1)$ has a generic point (e.g. see \cite{Denker}*{Corollary~21.15}). But if $x$ is a generic point for $\nu$ then 
$$\mathfrak{d}(\nu,\nu_1)=\mathfrak{d}(\lim\limits_{m\to\infty}\SE_m(x), \nu_1)=\lim\limits_{m\to\infty}\mathfrak{d}(\SE_m(x), \nu_1)\leq \ka_1$$ which implies $\mathfrak{d}(\nu,\nu_1)\leq2\ka_1$.

The claim is proved.
\end{proof}

\begin{claim}
  The following inequalities hold $$ (1+\delta_1)^2 c\int \rho d\mu \leq h_\nu(\sigma)=h_{top}(X_1)
\leq
(1+3\delta_1)c\int \rho d\mu.$$
   
\end{claim}
\begin{proof}
 
Observe that by the construction of $\hat{X}_1$, $\si^k|_{\hat{X}_1}$ is topologically conjugated to the full-shift on $\Ga_n$ symbols.  But this implies $h_{top}(X_1)=\frac{1}{k}\log(\#\Ga_n)$.
by Claim 1 we obtain that
$$
 |h_{\nu}(\si)-h_{\nu_1}(\si)|=\left|h_{top}(X_1)-h_{\nu_1}(\si)\right|=\left|\frac{1}{k}\log(\#\Ga_n)-h_{\nu_1}(\si)\right|\leq \ka_1$$
 and this concludes the proof by \eqref{int:est:eta}.
\end{proof}

 After finishing with the construction of $X_1$ we proceed by induction.
We simply use $X_n$ in the place of $\Sigma$ and measure of maximal entropy on $X_n$ in place of $\mu$. Then we proceed as before constructing shift of finite type $X_{n+1}$ keeping the following 
requirements.

 First, we select $0<\de_{n+1}<\de_n$ so that $(1+\de_{n+1})^2<(1+\de_n)$.  Since
$$
(1+\delta_{n+1})c\int \rho d\eta<
(1+\delta_{n+1})^2 c\int \rho d\nu_{n} \leq 
h_{top}(X_{n})
$$
for any measure $\eta$ sufficiently close to $\nu_{n}$.
This leads to a choice of $\kappa_{n+1}$ and a neighborhood
$$\{\eta : \mathfrak{d}(\nu_n,\eta)<\kappa_{n+1}/2\}$$ 
to proceed with Lemma \ref{lemma2}. We require additionally that $\kappa_{n+1}<\kappa_n/2$
to ensure convergence to $0$.
Repeating all previous arguments, we can construct a subshift of finite type $X_{n+1}\subset X_{n}$ satisfying:
$$
(1-\delta_{n+1})\int \rho d\nu_n < \int \rho d\eta < (1+\delta_{n+1})\int \rho d\nu_n 
$$
for any $\eta\in M_\sigma(X_{n+1})$ and
$$
(1+\delta_{n+1})^2 c\int \rho d\nu_n \leq 
h_{\nu_{n+1}}(\sigma)=h_{top}(X_{n+1})
\leq
(1+3\delta_{n+1})c\int \rho d\nu_n.
$$

To ensure ``furthermore'' part assume that $s_n$ is already give. 
Since $\nu_n$ is fully supported, there is $\xi>0$ such that $\nu_n([w])>2\xi$
for every $w\in \mathcal{L}_{s_n}(X_n)$. By definition of weak* topology, if $\mathfrak{d}(\nu_n,\eta)<2\kappa_n$, with $\kappa_n$ sufficiently small, then $\eta([w])>\xi$.
But we have already proved that for every $i\geq n$ we have $\mathfrak{d}(\nu_n,\nu_i)<2\kappa_n$
and therefore $\nu_i([w])>0$ for every $w\in \mathcal{L}_{s_n}(X_n)$, proving $\mathcal{L}_{s_n}(X_n)\subset \mathcal{L}_{s_n}(X_i)$ and therefore $\mathcal{L}_{s_n}(X_n)= \mathcal{L}_{s_n}(X_i)$
since $X_i\subset X_n$. On the other hand, $\nu_n$ is fully supported, ergodic and by induction 
$\mathcal{L}_{s_{n-1}}(X_{n-1})= \mathcal{L}_{s_{n-1}}(X_n)$, therefore taking $s_n$ sufficiently large, we know that for generic point $x$ for $\nu_n$, the word $x_{[0,s_n/2)}$ contains all elements of $\mathcal{L}_{s_{n-1}}(X_n)$ as subwords. Taking $N_1>2s_n$ in construction in Claim~1 completes the proof of this part.
The proof is complete.
\end{proof}

As an immediate consequence of Theorem \ref{thm: SFT}, we prove the following corollary:

\begin{corollary}\label{cor: expshad}
    If $\phi$ is a continuous expansive flow with the shadowing property, then $\phi$ has entropy flexibility.
\end{corollary}
\begin{proof}
    Suppose $\phi$ is expansive and has the shadowing property. By \cite{Mori}*{Theorem A}, $\phi$ is the factor of a suspension flow $\psi$ of a shift of finite type. Moreover, the quotient map is finite-to-one,  which in particular means that $\phi_1$ is a finite-to-one factor of $\psi_1$. On the other hand, the flow $\psi$ has entropy flexibility by Theorem \ref{thm: SFT}. Since by the classical result of Bowen \cite{Bo} finite extensions preserve entropy, it is easy to see that $\phi$ also has entropy flexibility.
\end{proof}

\section{Entropy Flexibility of Smooth Vector Fields}\label{sec: smoothflows}

In this section, we shall investigate the entropy flexibility of smooth vector fields. Our goal here is to provide the reader with the proofs of Theorem \ref{thm: variousflows} and \ref{thm: generic}.  Let us begin by recalling some basic concepts and facts of the theory of smooth flows.

\subsection{Preliminaries on Flows}

Throughout this section, let $M$ be a compact Riemannian manifold endowed with a Riemannian metric $\Vert\cdot\Vert$. We denote by $X$ a $C^1$ vector field on $M$, and by $\phi$ the flow generated by $X$. Let $\SX^1(M)$ be the set of $C^1$-vector fields over $M$ endowed with the $C^1$-topology. 
A point $x\in M$ is called a \emph{singularity} of $X$ if $X(x) = 0$. The set of all singularities of $X$ is denoted by $\operatorname{Sing}(X)$.
A compact $\phi$-invariant set $\Lambda$ is said to be \emph{attracting} if there exists a neighborhood $U_0$ of $\Lambda$, called the trapping region of $\Lambda$, such that $\overline{\phi_t(U_0)}\subset U_0$ for all $t > 0$, and
$$
\Lambda = \bigcap_{t\geq 0} \phi_t(U_0).
$$
Moreover, $\Lambda$ is called an \emph{attractor} if it is also transitive, i.e., there exists $z\in\Lambda$ such that $\omega(z) = \Lambda$.

Let us then recall the concept of hyperbolic set.

\begin{definition}
A compact invariant set $ \Lambda \subset M $ is called a \emph{hyperbolic set} for the flow $ \phi $ if for every $ x \in \Lambda $, the tangent space admits a continuous   $D\phi_t $-invariant splitting
$
T_x M = E^s_x \oplus \langle X(x) \rangle \oplus E^u_x,
$
where:
\begin{enumerate}
  \item $ \langle X(x) \rangle $ denotes the one-dimensional subspace generated by the vector field $ X $;
  \item the subbundles $ E^s $ and $ E^u $ are invariant under the derivative cocycle, that is, $ D\phi_t(x)E^s_x = E^s_{\phi_t(x)} $ and $ D\phi_t(x)E^u_x = E^u_{\phi_t(x)} $ for all $ t \in \mathbb{R} $;
  \item there exist constants $ K> 0 $ and $ \lambda > 0 $ such that for all $ t \geq 0 $,
  $$
  \|D\phi_t(x)|_{E^s_x}\| \leq K e^{-\lambda t}, \quad \text{ and } \quad \|D\phi_{-t}(x)|_{E^u_x}\| \leq K e^{-\lambda t}.
  $$
\end{enumerate}
\end{definition}
Observe that if $x\in \Sing(X)$, then $X(x)=0$, and therefore the direction $\langle X(x)\rangle=0$.
 In particular, it means that singularities must be isolated points in $\Lambda$.
Next, we recall the concept of dominated splitting, which plays a fundamental role in the study of non-uniformly hyperbolic dynamics.

\begin{definition}
 A compact invariant set $\Lambda$ is said to have a \emph{dominated splitting} if there exists a continuous $D\phi_t$-invariant splitting $T_{\Lambda}M = E \oplus F$ and constants $K, \lambda > 0$ such that 
$$
\frac{\Vert D\phi_t(x)\vert_{E_x} \Vert}{m(D\phi_t(x)\vert_{F_x})} \leq Ke^{-\lambda t}, \quad \forall x\in\Lambda, \ \forall t>0,
$$
where $m(D\phi_t(x))$ denotes the co-norm of $D\phi_t(x)$, i.e., $$m(D\phi_t(x)|_{F})=\inf\{||D\phi_t(x) v||; v\in F, ||v||=1\}.$$ In this case, we say that $F$ is \emph{dominated} by $E$.

The set $\Lambda$ is called \emph{partially hyperbolic} if the subbundle $E$ is of contracting type, that is,
$$
\Vert D\phi_t(x)\vert_{E_x} \Vert \leq Ke^{-\lambda t}
$$
for every $t > 0$ and $x \in \Lambda$.  We call $F$ the \textit{central bundle} in that case. 
  \end{definition}

Since hyperbolic behavior does occur in the vector field direction, it is useful to analyze the dynamics transverse to the flow direction. For this purpose,  we make use of the \emph{linear Poincaré flow}, which projects the derivative flow onto the normal bundle of the vector field. If $X$ is $C^1$-vector field on a manifold $M$ and $\phi$ denote $M'=M\setminus \Sing(X)$. Define the normal bundle of $X$ as the vector bundle  
$$\SN=\{(x,v)\in TM'; v\perp X(x) \}.$$
\begin{remark}
If $\Lambda$ is a compact and invariant subset of $M$, we denote $\Lambda'=\Lambda\setminus \Sing(X)$. Then we define the normal  bundle $\SN^{\Lambda'}$ analogously to $\SN$.    
\end{remark}

\begin{definition}
Let $X$ be a $C^1$ vector field on a manifold $M$ and let $\phi$ denote its associated flow. 
The \emph{linear Poincaré flow} $\psi : \R\times \mathcal{N} \to \mathcal{N}$ is defined by
$$
\psi_t(x,v) = (\phi_t(x),\Pi_{\phi_t(x)} \circ D\phi_t(x)v),
$$
where $\Pi_{\phi_t(x)} : T_{\phi_t(x)}M \to \mathcal{N}_{\phi_t(x)}$ denotes the orthogonal projection onto $\mathcal{N}_{\phi_t(x)}$.
\end{definition}

Consider a compact invariant set $ \Lambda \subset M$ and assume that the tangent bundle over $ \Lambda $ admits a continuous $ D\phi_t$-invariant dominated splitting
$
T_\Lambda M = E \oplus F.$ 
Let $\mathcal{N}^{\Lambda}$  be the normal bundle over $ \Lambda'$.  If the flow direction is contained in one of the bundles $E$ or $F$, then one can define the projected sub-bundles
$
\widetilde{E}_x = \Pi_x(E_x), \quad \widetilde{F}_x = \Pi_x(F_x),
$
and obtain a dominated splitting for the linear Poincaré flow:
$$
\mathcal{N}^{\Lambda} = \widetilde{E} \oplus \widetilde{F},
$$
which is $ \Psi$-invariant and dominated in the same sense:
$$
\frac{\|
{ \psi}_t|_{\widetilde{E}_x}\|}{  m(
{ \psi}_{-t}|_{\widetilde{F}_{ 
x
}})} \leq C e^{-\lambda t}, \quad \text{for all } t \geq 0.
$$
Therefore, the dominated splitting of the tangent bundle can be lifted to a dominated splitting of the linear Poincaré flow on the normal bundle, provided the flow direction is contained in one of the subbundles.

We now introduce the classes of flows appearing in our results. 
\begin{definition}
A vector field $X$ is called a \emph{star flow} if there exists a $C^1$-neighborhood $\mathcal{U}$ of $X$ such that every $Y \in \mathcal{U}$ has all its periodic orbits and singularities hyperbolic. 
\end{definition}

\begin{definition}\label{def:46}
Let $\Lambda$ be a compact invariant partially hyperbolic set of a vector field $X$ whose singularities are hyperbolic. We say that $\Lambda$ is \emph{asymptotically sectional-hyperbolic}  if its central subbundle $F$ is eventually asymptotically expanding outside the stable manifolds of the singularities, that is, there exists $C > 0$ such that
\begin{equation}\label{ash}
\limsup_{t \to +\infty} \frac{1}{t} \log \left| \det\left( D\phi_t(x)\vert_{L_x} \right) \right| \geq C,
\end{equation}
for every $x \in \Lambda' = \Lambda \setminus W^s(\operatorname{Sing}(X))$ and every two-dimensional subspace $L_x \subset F_x$. 

We say that an ASH set is \emph{non-trivial} if it is not reduced to a singularity.
\end{definition}

\subsection{Smooth ergodic theory of vector fields.}
\medskip

We now recall some basic notions from differentiable ergodic theory of flows.
\begin{definition}
Let $X$ be a $C^1$ vector field over a compact Riemannian manifold $M$ with associated flow $\phi$. Fix $x\in M$ and $v\in T_xM\setminus\{0\}$. We say that a number $\lambda(x,v) \in \mathbb{R}$ is the \emph{Lyapunov exponent} of $X$ at $x \in M$ in the direction $v$ if 
$$
\lim_{t \to \pm \infty} \frac{1}{t} \log \Vert D\phi_t(x)v \Vert = \lambda(x,v).
$$
\end{definition}

\medskip

 Lyapunov exponents do not always exist. Nevertheless their existence almost everywhere is guaranteed by the classical Oseledets’ theorem, which we now state.

\begin{theorem}[Oseledets' Theorem]
Let $\mu$ be an invariant probability measure for a $C^1$-flow $\phi$ over a compact manifold $M$. Then, there exists a $\phi_t$-invariant set $\Lambda \subset M$ with $\mu(\Lambda) = 1$ such that for every $x \in \Lambda$, there exists a measurable splitting
$$
T_xM = \bigoplus_{i=1}^{k(x)} E^i_x,
$$
and real numbers (Lyapunov exponents) $\lambda_1(x) \leq \cdots \leq \lambda_{k(x)}(x)$ satisfying
$$
\lim_{t \to \pm \infty} \frac{1}{t} \log \Vert D\phi_t(x)v \Vert = \lambda_i(x) \quad \text{for all } v \in E^i_x \setminus \{0\}.
$$
Moreover, when $\mu(\Sing(X))=0$, one of the exponents is always zero and corresponds to the direction of the flow: $X(x) \in E^i_x$ for some $i$ with $\lambda_i(x) = 0$.
\end{theorem}

\begin{definition}
Let $\mu$ be an invariant probability measure for a $C^1$-flow $\phi$. We say that $\mu$ is a \emph{hyperbolic measure} if  there is full-measure set $R\subset M$ such that for every $x\in R$ and every $v\notin \langle X(x)\rangle$, we have $\lambda(x,v)\neq 0.$ 

\end{definition}

Notice that when $\mu$ is hyperbolic, then for almost every regular point $x\in M$, its Oseledets splitting is of the form:
$$T_xM=E^1_x\oplus\cdots \oplus E^{i_s}_x\oplus \langle X(x)\rangle \oplus E^{i_u}_x\cdots E^{k(x)}_x,$$
with corresponding Lyapunov exponents $$\lambda_1(x)\leq\cdots \lambda_{i_s}(x)<\lambda_{i_s+1}(x)=0<\lambda_{i_u}(x)\leq\cdots\lambda_{k(x)}(x).$$
It is well-known that if $\mu$ is an ergodic hyperbolic measure, then $k(x)$, $\lambda_i(x)$ and $\dim E_i(x)$ are constant, for $\mu$-almost every $x\in M$ and every $i=1,...k$. In this case, we define the {\it index of $\mu$} as: $$\Ind(\mu)=\sum_{i=1}^{i_s}\dim E_i,$$ where $\lambda_{i_s}$ is the biggest negative Lypaunov exponent of $\mu$.

\begin{theorem}[Ruelle's inequlity]
Let $X$ be a  $C^1$-vector field on a compact Riemannian manifold $M$. If $\mu$ is a $\phi$-invariant ergodic probability measure, then

$$h_\mu(\phi) \leq \int_M \sum_{\lambda_i(x) > 0} \lambda_i(x) \cdot \dim E_i(x) \, d\mu(x),
$$

\end{theorem}

\subsection{Proof of Theorem \ref{thm: variousflows} and \ref{thm: generic}}
Now we are ready to prove Theorems \ref{thm: variousflows} and \ref{thm: generic}. To begin with, let us introduce the following technical condition. 
\begin{definition}\label{condD}
    Let $\mu$ be an ergodic hyperbolic measure for $X$. We say that $\mu$ satisfies the condition $(D)$ if its linear Poincar\'e flow has a dominating splitting in  $\supp(\mu)\backslash \Sing(X)$ and $\dim(E)=\Ind(\mu)$. 
\end{definition}

Next, we derive the following Lemma that will play a central role in our proofs. The following is an easy consequence of \cite{LW}*{Proposition 4.1}. We present the proof for completeness.
\begin{lemma}
\label{t.1}
Let $X$ be a $C^1$-vector field over $M$ inducing a flow $\phi$ with positive entropy. Suppose there is a sequence $\mu_n$ of hyperbolic ergodic measures  
 for $\phi$ satisfying condition  (D)  such that $$h_{\mu_n}({ \phi}
)\to h_{top}({ \phi}
).$$
Then for every $\eps>0$, there is a horseshoe $\Gamma_{\eps}$ { for $\phi$} such that $|h_{top}({ \phi}
|_{\Gamma_\eps})-h_{top}({ \phi}
)|\leq \eps$. 
\end{lemma}

\begin{proof}
    Let  $\phi$ be a flow induced by a $C^1$-vector field $X$
    over $M$ and $\mu_n$ be a sequence of hyperbolic measures such that $h_{\mu_{n}}({ \phi}
    )\to h_{top}({ \phi}
    )$. Suppose that those measures satisfy Definition \ref{condD}. Fix $0<\eps<\frac{h_{top}({ \phi}
    )}{2}$ and  $n>0$ such that $$|h_{\mu_n}({ \phi}
    )-h_{top}({ \phi}
    )|\leq \frac{\eps}{2}.$$

    Since $h_{top}({ \phi}
    )>0$, we obtain $h_{\mu_n}({ \phi}
    )>0$ and therefore $\mu_n(\Sing(X))=0$. To conclude the proof we apply \cite{LW}*{Proposition 4.1} and obtain a horseshoe $\Gamma_{\eps}\subset M$ such that $|h_{top}({ \phi}|_{\Gamma_{\eps}})-h_{\mu_n}({ \phi}
    )|\leq\frac{\eps}{2}$. Therefore 
    $$|h_{\mu_n}({ \phi}
    )-h_{top}({ \phi}
    )|\leq\eps$$
    completing the proof.
\end{proof}

Now we are able to provide a proof for Theorem \ref{thm: variousflows}.

\begin{proof}[Proof of Theorem \ref{thm: variousflows}]
Let $X$ be a vector field over $M$  and let $\phi$ be the associated flow. 
Without loss of generality, suppose $h_{top}({ \phi}
)>0$. Firstly, suppose that ${ \phi}
$ satisfies the assumptions of Lemma \ref{t.1}. So, for every $\eps>0$, there is a horseshoe $\Gamma_\eps\subset M$ satisfying $$|h_{top}({\phi}
|_{\Gamma_\eps})-h_{top}( \phi)
)|\leq \eps.$$ Hence, by applying Theorem \ref{thm: SFT} to $\Gamma_{\eps}$, we obtain for every $0\leq c <h_{top}({ \phi}
|_{\Gamma_\eps})$  
{ a $\phi$-invariant set }
$\Gamma_c\subset \Gamma_{\eps}$ 
 such that the flow $\phi|_{\Lambda_c}$ is strictly ergodic, and its unique invariant measure $\mu_c$ satisfies
$$h_{top}({ \phi}
|_{\Gamma_c})=h_{\mu_c}({ \phi}
)=c.$$
Since $h_{\mu_n}({ \phi}
)\to h_{top}({ \phi}
)$,  we can find such set $\Gamma_c$ for any $c\in [0,h_{top}(\phi))$, thus we  conclude that ${ \phi}
$ has entropy flexibility.

To prove
\eqref{thmC:a}, we observe that it was proven in \cite{SGW}*{ Theorem~E} that every ergodic measure of a star flow $\phi$ is hyperbolic. By the variational principle, let $\mu_n$ be a sequence of ergodic measures so that $h_{\mu_n}(\phi)\to h_{top}(\phi)$. By  \cite{LSWW}*{Lemma~2.6}, these measures satisfy the conditions in Definition~\ref{condD}. Therefore,  we obtain that the star flows  satisfy the assumptions of Lemma \ref{t.1} and so the result follows.

Finally, let us prove \eqref{thmC:b}.
Assume that $\Lambda$ is an ASH attracting set and note that the hyperbolicity of the ergodic measures of $\phi|_{\Lambda}$ was obtained in \cite{APRV}*{Theorem 3.2}.  By definition, $T_\Lambda M=E\oplus F$ is a dominated splitting with the subbundle $E$ being of contracting type, therefore $X(x)\in F$ for every $x\not\in \Sing(X)$ and so the dominated splitting  be lifted to its linear Poincar\'e flow outside singularities. 
Next, let us show $\dim(E)=\Ind(\mu)$, for every ergodic measure with positive entropy.  Let $R$ be a full-measure set given by Oseledets' Theorem for $\mu$. Indeed, since $\mu$ has positive entropy, then $\mu(\Sing(X))=0$, and hence we can assume $R$ does not contain singularities. Since $E$ is of contracting type, for every $v\in E\setminus \{0\}$ and every $x\in R$, we have $\lambda(x,v)<0$ and therefore, $\dim(E)\leq \Ind(\mu)$. 

Let $x\in R$ and $v\in F\setminus\langle X(x)\rangle$. Since $v\notin \langle X(x) \rangle$,  the subspace $L\subset F$ generated by $X(x)$ and $v$ is two dimensional. Observe that:

$$
\limsup_{t \rightarrow \infty} \frac{1}{t} \log | \det D\phi_{t}(x)|_{L}|=
\limsup_{t \rightarrow \infty} \frac{1}{t} \log \left|  \frac{\sin \theta_t\cdot\Vert X(\phi_t(x))\Vert\cdot\Vert D\phi_t(x)v\Vert}{\Vert X(x)\Vert}\right|,$$
where $\theta_t$ denotes the angle between $X(\phi_t(x))$ and $D\phi_t(x) v$. Hence, since $\Lambda$ is asymptotically sectional-hyperbolic,  a straightforward computation shows  
$$
  0<C<\limsup_{t \rightarrow \infty} \frac{1}{t} \log |\det D\phi_{t}(x)|_{L}|= \lambda(x,X(x))+\lambda(x,v).
$$
Since $\lambda(x,X(x))=0$, we conclude $\lambda(x,v)>
0$.
Finally, if $v\notin E,F$  is a non-zero vector, there are $a,b\neq 0$ so that $v=av_E+bv_F$, with $v_E\in E$ and $v_F\in F$. By the basic properties of Lyapunov exponents, we obtain $$\lambda(x,v)=\max\{\lambda(x,v_E),\lambda(x,v_F)\}>0 $$
and therefore $dim(E)=Ind(\mu)$. So the conclusion follows immediately from the variational principle and theorem \ref{t.1}.
\end{proof}

Next, we proceed to prove Theorem \ref{thm: generic}.

\begin{proof}[Proof of Theorem \ref{thm: generic}]
 Assume first that $M$ is a three-dimensional manifold.
    Let us consider the following decomposition $$\SX^1(M)=\SA_0\cup\SA,$$ 
    where $\SA_0$ denotes the vector fields with zero topological entropy, while $\SA$ denote the vector fields with positive entropy. 
   We will first consider vector fields in $\SA$.

 Fix any vector field $X\in \SX^1(M)$ such that associated flow $\phi$ satisfies $h_{top}(\phi)>0$.
    Let $L=\|X\|_{C^1}+1$.  Recall that since $M$ is compact, we have $h_{top}({ \phi}
    )\leq \log(L)<L$.  Thus, there is a $C^1$-neighborhood $\SU_X$ of $X$ such that for any $Y\in \SU_X$ and its associated flow $\psi$ we have $h_{top}(\psi)\leq 3L$.
Given $Z\in \SX^1(M)$ with generated flow $\phi^Z$,  consider the following condition:
\begin{eqnarray}
    &&\text{$\Gamma$ is any hyperbolic horseshoe of $\phi^Z$,}\label{cond:star}\\ 
    &&\quad\quad\text{$\mu$ is an ergodic for $\phi^Z$ and $\supp(\mu)=\Gamma$.\nonumber}
\end{eqnarray}
 Extend $\SK(M)$ to a compact metric space $\SK^*(M)$ by adding an isolated point $\SK^*(M)=\SK(M)\cup \{\emptyset\}$.
For every $Z\in \mathcal{U}_X$ consider the set
$\mathcal{V}_Z\subset \SK(\SM(M)\times \SK^*(M)\times [0,3L])$ consisting of triples $(\mu, \Gamma,h)$
such that $h_\mu(\phi^Z|_\Gamma)=h$ and condition \eqref{cond:star} is satisfied. We also assume for technical reasons that $\Gamma=\emptyset$ satisfies \eqref{cond:star}. Then 
$$
\mathcal{V}_Z\supset \SM_{\phi^Z}(M)\times \{\emptyset\}\times \{0\}\neq \emptyset.
$$

By the above definitions, the following auxiliary map is well defined:
$$
\Phi:\SU_X\ni Z\mapsto \overline{\mathcal{V}_Z}\in \SK( \SM(M)\times \SK^*(M)\times [0,3L])
$$

Recall from the classical hyperbolic theory that hyperbolic horseshoes, i.e., horseshoes that are also hyperbolic sets, are stable. In particular, this implies that if $Z$ has a hyperbolic horseshoe $K_Z$, any $C^1$-close vector field $Y$ also has a horseshoe $K_Y$ that is close to $K_Z$ in the Hausdorff topology. Moreover, the entropy of $K_Y$ varies continuously with $Y$. Consequently, by possibly reducing $\SU_X$,  the map $\Phi$ is lower semi-continuous, i.e., for every $Z\in \SU_X$, every  $Z_n\to Z$ in the $C^1$-topology and every $(\mu,\Gamma,h)\in \Phi(Z)$, there is a sequence  one has $(\mu_n,\Gamma_n,h_n)\in \Phi(Z_n)$ such that 
$$(\mu_n,\Gamma_n,h_n)\to (\mu,\Gamma,h).$$ 
So by \cite{Pu}*{Proposition 6.2}, there is  a dense $G_\delta$ subset $\SR_X\subset \SU_X$ such that $\Phi|_{\SR_X}$ is continuous. 

Let $Z\in \SR_X$. If $h_{top}(\phi^Z)=0$, then $Z\in \SA_0$. So, let us consider $Z\in \SA\cap \SR_X$.  Fix $\eps>0$ and by the variational principle, let $\mu$ be an ergodic measure of $Z$ such that $h_{\mu}(\phi^Z)>h_{top}(\phi^Z)-\eps$. In particular, $\mu(\Sing(X))=0$. Observe that by the ergodicity of $\mu$, their exponents are almost everywhere constant. So, by the Ruelle's inequality 
$$0<h_\mu(\phi^Z) \leq \int_M \sum_{\lambda_i(x) > 0} \lambda_i(x) \cdot \dim E_i(x) \, d\mu(x).
$$
Consequently, there is a positive exponent $\lambda^+$ associated to a sub-bundle $E^+$. On the other hand, an analogous reasoning for the reverse flow generated by $-X$, we obtain the existence of a negative exponent $\lambda^-$ associate with a sub-bundle $E^-$. Since $\dim(TM)=3$ and $\lambda( x, X(x))=0$, for every regular point, we obtain that $\mu$ is hyperbolic.
As a by-product of the above argument, we have that
\begin{equation}
\label{eq:Ru:exp} h_\mu(\phi^Z) \leq \min\{|\lambda^+|,|\lambda^-|\}.
\end{equation}

We then split the proof into two cases:

\textbf{Case 1:} Suppose, the Oseledec splitting of $\mu$ is dominated then there exists a horseshoe  $\Gamma_{\eps}$ (see for instance \cite{LW}*{Proposition~4.1}) and $\delta>0$ such that 
$$h_{top}(\phi^Z|_{\Gamma_{\eps}})>h_{\mu}(\phi^Z)-\delta>h_{top}(\phi^Z)-\eps.$$

\textbf{Case 2:} If the Oseledec splitting of $\mu$ is not domminated, we procced as follows. By Ma\~n\'e's ergodic closing lemma, we can suppose that for any $Z\in \SR_X$ and any ergodic measure $\mu$ not supported on singularities, there exists a sequence of periodic orbits $O_n$ 
 such that the invariant measures $\nu_n$ defined by $O_n$ converge to $\mu$ in the weak* topology and also their Lyapunov exponents converges to the exponents of $\mu$ ( it is enough to approximate sufficiently long initial segment of generic point of $\mu$; see for instance proof of \cite{AS}*{Theorem 3.5})

By \eqref{eq:Ru:exp}, for every $\delta>0$ there is $n$ such that
the Lyapunov exponents $\lambda^{\pm}_n$ of the orbit $O_n$  satisfy:
$|\lambda^{\pm}_n|>h_{\mu}(\phi^Z)-\delta.$
  
 If the union of the hyperbolic
periodic orbits $\{O_n\}_n$ admit a dominated splitting, then we have reduction to Case~1. Therefore, assume the contrary. Since the periods of the periodic orbits $O_n$ are unbounded, 
\cite{LW}*{Theorem 3.6} yields a sequence of vector fields $Y_n\in \SX^1(M)$, $Y_n\to Z$ and horseshoes $\Gamma_n$ for $\phi^{Y_n}$ such that $\Gamma_n$ belongs to small neighborhoods of $O_n$ and
$$h_{top}(\phi^{Y_n}|_{\Gamma_n})>h_{\mu}(\phi^Z)- \frac{\delta}{2}.$$
 If we select each $\Gamma_n$ sufficiently close to the hyperbolic set $O_n$, they will be themselves hyperbolic. By Theorem~\ref{thm: reduction2}
we can obtain ergodic measure $\mu_n$ of $\phi^{Y_n}$
such that $h_{\mu_n}(\phi^{Y_n})>h_{\mu}(\phi^Z)-\delta$ 
and $\supp(\mu_n)=\Lambda_n\subset \Gamma_n$
where $\Lambda$ is a horseshoe (obtained as a suspension of shift of finite type from a sequence in Theorem~\ref{thm: reduction2}).
Recall that $Z$ is a continuity point of  $\Phi$ (and $\delta>0$ is arbitrary), thus we obtain (going to a subsequence when necessary) a hyperbolic horseshoe $\Gamma_{\eps}$ arbitrarily close in the Hausdorff distance to $\lim_{n\to \infty} \Lambda_n$  
and an ergodic measure $\nu$ arbitrarily close to $\lim_{n\to \infty} \mu_n\in \SM(M)$ with $\supp(\nu)=\Gamma_\eps$ and such that 
$$h_{top}(\phi^Z|_{\Gamma_{\eps}})\geq h_\nu(\phi^Z|_{\Gamma_{\eps}})>h_{\mu}(\phi^Z)-\delta>h_{top}(\phi^Z)-\eps.$$

But $\Gamma_\eps$ is a suspension flow over full shift, hence Theorem \ref{thm: SFT} ensures flexibility for $Z$, hence for any $Z\in\SR_X$. Since $\SR_X$ is a residual subset of $\SU_X$, then it can be written as $$\SR_X=\bigcap_{n\in \N} \SU_{X,n},$$
where $\SU_{X,n}$ is an open and dense subset of $\SU_X$. Fix a countable, dense in $\SA$ sequence $X_n\in \SA$. Let $\mathcal{U}_1=\mathcal{U}_{X_1}$. Now assume that pairwise disjoint sets $\mathcal{U}_1, \ldots, \mathcal{U}_n$ have been constructed and let $m$ be the smallest index, such that $X_m\not\in \overline{\mathcal{U}_1}\cup \ldots \cup \overline{\mathcal{U}_n}$ and then let $X_m\in \mathcal{U}_{n+1}\subset \mathcal{U}_{X_m}$
be any open set disjoint with each $\mathcal{U}_j$ for $j<n$. After completing the induction, we obtain a sequence of pairwise disjoint sets $\mathcal{U}_n$
such that $\SA \subset \overline{\bigcup_{n=1}^\infty \mathcal{U}_n}$.
Let $\mathcal{U}_0=\SA_0\setminus \overline{\bigcup_{n=1}^\infty \mathcal{U}_n}$. There are indexes $i(n)$ that $\mathcal{U}_n=\mathcal{U}_{X_{i(n)}}$.
Observe that (by disjointness of sets $\mathcal{U}_n$, we have
$$\SR=\bigcap_{k\in \N}\bigcup_{n=1}^\infty\SU_{X_{i(n)},k}=
\bigcup_{n=1}^\infty\bigcap_{k\in \N}\SU_{X_{i(n)},k}=\bigcup_{n=1}^\infty \SR_{X_{i(n)}}.$$ This proves that $\SR$ is residual in $\overline{\bigcup_{n=1}^\infty \mathcal{U}_n}$ and as a consequence $\SR\cup \mathcal{U}_0\subset \SR\cup \SA_0$ is residual in $\SX^1(M)$. This completes the proof of \eqref{thm:generic:1}.

Next, suppose that $M$ is a closed surface and denote by $\Diff^1(M)$ the set of $C^1$-diffeomorphisms on $M$. To conclude the proof, we need to show that there is a residual set $\SR\subset \Diff^1(M)$ such that every $f\in \SR$ has entropy flexibility.  To construct the set $\SR$, we can  repeat verbatim the above proof for flows, by doing the obvious adaptations to the diffeomorphism case. Indeed, all the previous construction holds in the discrete-time case, and we just need to use the diffeomorphism versions of \cite{LW}*{Theorem 3.6 and Proposition 4.1}, which are respectively, \cite{WL}*{Theorem 4.6 and Theorem 4.7}.

\end{proof}

\section{Applications to Partially Hyperbolic Diffeomorphisms}\label{sec: PH}

This section is devoted to discussing the entropy flexibility of discrete-time dynamical systems. In Theorem \ref{thm: generic}, we have seen that entropy flexibility holds $ C^1$-generally for surface diffeomorphisms. Here we will explore it for partially hyperbolic diffeomorphisms. Throughout this section, let $M$ be a compact Riemannian manifold with Riemannian metric $\Vert\cdot\Vert$, and let $f: M \to M$ be a $C^1$ diffeomorphism. We begin by recalling the fundamental notion of dominated splitting. We say that  $f$ admits a \emph{dominated splitting} if there exists a continuous $Df$-invariant splitting 
$$
T_\Lambda M = E^1 \oplus E^2 \oplus \dots \oplus E^k,
$$
and constants $K_i > 0$ and $\lambda \in (0,1)$ such that, for every $x \in M$, all $n \geq 0$, and every $1 \leq i   \leq k-1$, one has 
$$
\frac{\|Df^n|_{E^i_x}\|}{m(Df^n|_{E^{i+1}_x})} \leq K_i \lambda^n,
$$

\begin{definition}
A diffeomorphism $f$ is said to be \emph{partially hyperbolic} if there exists  a splitting is dominated. $Df$-invariant splitting
$
T_\Lambda M = E^s \oplus E^c \oplus E^u,
$
such that:
\begin{itemize}
    \item The bundle $E^s$ is uniformly contracting: there exist constants $C > 0$ and $\lambda \in (0,1)$ such that
    $
    \|Df^n|_{E^s_x}\| \leq C \lambda^n \quad \text{for all } x \in \Lambda, \, n \geq 0;
    $
    \item The bundle $E^u$ is uniformly expanding: there exist constants $C > 0$ and $\lambda \in (0,1)$ such that
    $
    \|Df^{-n}|_{E^u_x}\| \leq C \lambda^n \quad \text{for all } x \in \Lambda, \, n \geq 0;
    $
\end{itemize}
\end{definition}

Let $f:M \to M$ be a partially hyperbolic diffeomorphism with splitting $E^s\oplus E^c\oplus E^u$.  Remember from the invariant manifold theory \cite{HPS} that there exists an invariant stable foliation $\SF^s$ and an invariant unstable foliation $\SF^u$ which are tangent to the bundles $E^s$ and $E^u$, respectively. On the other hand, the existence of an invariant foliation $\SF^c$ tangent to the center bundle $E^c$ cannot always be granted.  Here we will suppose that there exists a invariant center foliation $\SF^c$ tangent to $E^c$. 

We say that the center foliation is {\it compact} if all of its leaves are compact. Moreover, $\SF^c$ is {\it uniformly compact} if every leaf of $\SF^c$ has finite holonomy. In this case,  the quotient space $\tilde{M}=M/\SF^c$ is a compact metric space, using the Hausdorff distance between center leaves. Let us call $\pi: M\to \tilde{M}$ the quotient map and $F$ the homeomorphism induced in $\tilde{M}$ (see for instance \cite{Eps}). We are now in position to state our first application.

\begin{theorem}
\label{c.partial}
Let $f: M\to M$ be a partially hyperbolic diffeomorphism with a one-dimensional center bundle. Suppose $\SF^c$ is uniformly a compact center foliation with trivial holonomy. Then $f$ has flexibility of the entropy. 
\end{theorem}
\begin{proof}
   Let $f$ be a partially hyperbolic diffeomorphism under the assumptions. By \cite{BB}*{Proposition 4.20 and Theorem 2}, the quotient map $F$ is expansive and has the shadowing property.
  Since $\dim(E^c)=1$, then the center foliation being uniformly compact is equivalent to the fact that leaves have uniformly bounded length. In particular, the topological entropy of the fibers of $\pi$ is zero. Hence,
$$0<h_{top}(f)=h_{top}(F).$$

 By the Spectral Decomposition Theorem (and variational principle), there is a chain recurrent class $H$ for $F$
such that $F|_H$ is expansive, transitive, has shadowing and $h_{top}(F)=h_{top}(F|_H)$,
e.g. see \cite{Aoki}*{Remark~4.2.9}. By a well-known result of Bowen (e.g. see \cite{Aoki}*{Theorem~4.3.6}) there is a shift of finite type $\Sigma_A$ and a finite-to-one semi-conjugacy $\eta \colon (\Sigma_A,\sigma)\to (H,F|_H)$.
Since $\eta$ preserves topological and metric entropy and ergodic measures of $\sigma$ lift to $F$, we obtain that $F$ has entropy flexibility by Corollary~\ref{cor: SFT}.

Since fibers of $\pi$ carry zero entropy, by \cite{Bo}, for any $$0\leq c\leq h_{top}(F)=h_{top}(f),$$ there exists a compact and invariant subset $\tilde{\Gamma_c}\subset\tilde{M}$ such that if we denote $\Gamma_c=\pi^{-1}(\tilde{\Gamma_c})$ then
$$h_{top}(f,\Gamma_c)=h_{top}(F,\tilde{\Gamma_c})=c.$$

To prove the existence of an ergodic measure $\mu_c$ satisfying $h_{\mu_c}(f)=c $, we recall that every partially hyperbolic set with $\dim(E^c)=1$ admits a measure of maximal entropy by \cite{DFPV}. So, we just need to take $\mu_c$ as a measure of maximal entropy for $\Gamma_c$.

\end{proof}
\begin{corollary}
    Let $f$ be a partially hyperbolic diffeomorphism with one-dimensional central direction. Assume, $dim(E^u)=1$ and $E^u$ is oriented  Then $f$ has entropy flexibility.
\end{corollary}

\begin{proof}
 To see why the corollary holds, just notice that  \cite{Boh}, If $f$ is a partially hyperbolic diffeomorphism with uniformly compact center foliation, $\dim(E^u)=1$  and $E^u$ is oriented then the center leaves have trivial holonomy.
\end{proof}

Next, we keep pursuing the entropy flexibility of diffeomorphisms and manage to deal with some partially hyperbolic systems. 
 It is clear that if we can obtain a sequence of compact and invariant sets $K_n$ that are 
 bounded-to-one factors of
shifts of finite type and such that $h_{top}(f|_{K_n})$ converge to $h_{top}(f)$. In particular, this condition is satisfied when the sets $K_n$ are hyperbolic horseshoes. This approximation property was obtained in several different contexts in the works  \cite{Zhang}, \cite{YZ}, \cite{DHSR}, and \cite{LOR}. Our last result collects all these results together and states that entropy flexibility holds in those contexts.  

\begin{corollary}
\label{c.zhang}
Let $f$ be a diffeopmorhism. If any of the following conditions are satisfied:
\begin{enumerate}
    \item $f$ is partially hyperbolic with one-dimensional center direction such that, any $u$-measure has positive center Lyapunov exponent and the strong stable foliation is minimal.
    \item $f$ robustly non-hyperbolic non-transitive diffeomorphism.
 
\end{enumerate}
Then $f$ has entropy flexibility.
\end{corollary}

\begin{remark}
Another situation where we can obtain flexibility of the entropy is the partially hyperbolic horseshoes constructed in Diaz et al \cite{DHSR}.  Indeed, there exists a partially hyperbolic structure with a one-dimensional center bundle. Those partially hyperbolic horseshoes were proven to factor over horseshoes. In addition,  Leplaideur et al \cite{LOR} prove that any invariant ergodic measure is hyperbolic. Thus, our arguments apply. We remark that those systems are at the boundary of uniformly hyperbolic diffeomorphisms.
\end{remark}

\section*{Acknowledgments.}
Alexander Arbieto was partially supported by CAPES, CNPq, PRONEX-Dynamical Systems and FAPERJ E-26/201.181/2022 “Programa Cientista do Nosso Estado from Brazil”.
Piotr Oprocha is grateful to Max Planck Institute for Mathematics in Bonn for its hospitality and financial support. 
\begin{table}[h]
\begin{tabularx}{\linewidth}{p{1.5cm}  X}
\includegraphics [width=1.8cm]{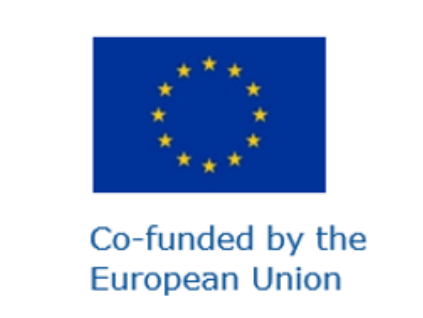} &
\vspace{-1.5cm}
This research is part of a project that has received funding from
the European Union's European Research Council Marie Sklodowska-Curie Project No. 101151716 -- TMSHADS -- HORIZON--MSCA--2023--PF--01.\\
\end{tabularx}
\end{table}

\end{document}